\documentclass[11pt]{amsart}

\usepackage{graphicx,overpic,amssymb,amscd,amsfonts,times}

\newcommand{\Rthree}{(\mathbb{R}^{3},\xi_{std})}
\newcommand{\Sthree}{(S^{3},\xi_{std})}

\newcommand{\SOT}{(S^{3},\xi_{OT})}
\newcommand{\Mxi}{(M,\xi)}
\newcommand{\AOB}{(\Sigma,\Phi)}

\newcommand{\OBOT}{(S^{1}\times[0,1], D^{-}_{S^{1}\times\lbrace\frac{1}{2}\rbrace})}
\newcommand{\disk}{\mathbb{D}}

\newcommand{\rsa}{\rightsquigarrow}

\newcommand{\be}{\begin{enumerate}}
\newcommand{\ee}{\end{enumerate}}

\newtheorem{thm}{Theorem}[section]

\newtheorem{prop}[thm]{Proposition}

\newtheorem{defn}[thm]{Definition}

\newtheorem{lemma}[thm]{Lemma}
\newtheorem{cor}[thm]{Corollary}

\newtheorem{rmk}[thm]{Remark}

\begin{document}

\title{Contact surgery and supporting open books}

\author{Russell Avdek}
\address{University of Southern California, Los Angeles, CA 90089}
\email{avdek@usc.edu}
\date{This version: May, 2011}

\begin{abstract}
Let $\Mxi$ be a contact 3-manifold.  We present two new algorithms, the first of which converts an open book $(\Sigma,\Phi)$ supporting $\Mxi$ with connected binding into a contact surgery diagram.  The second turns a contact surgery diagram for $\Mxi$ into a supporting open book decomposition.  These constructions lead to a refinement of a result of Ding-Geiges \cite{DG:LWall}, which states that every such $\Mxi$ may be obtained by contact surgery from $\Sthree$, as well as bounds on the support norm and genus \cite{EtOzb:Support} of contact manifolds obtained by surgery in terms of classical link data.  We then introduce Kirby moves called ribbon moves which use mapping class relations to modify contact surgery diagrams.  Any two surgery diagrams of the same contact 3-manifold are related by a sequence of Legendrian isotopies and ribbon moves.  As most of our results are computational in nature, a number of examples are analyzed.
\end{abstract}

\maketitle


\section{Introduction}

\subsection{Open books and contact structures}

A \emph{contact structure} on an oriented 3-manifold $M$ is a hyperplane distribution $\xi$ for which there exists a globally defined one-form $\alpha$ satisfying $\xi =$ Ker$(\alpha)$ and $\alpha\wedge d\alpha >0$ with respect to the prescribed orientation on $M$.  In this paper, we will consider two contact manifolds $\Mxi$ and $(N,\zeta)$ to be equivalent if they are diffeomorphic, i.e. if there is a diffeomorphism $\Psi:M\rightarrow N$ for which $T\Psi(\xi)=\zeta$.

An \emph{open book decomposition} of a closed, oriented 3-manifold $M$ is a pair $(B,\pi)$ consisting of an oriented link $B\subset M$ (called the \emph{binding}) and a fibration $M\setminus B \xrightarrow{\pi} S^{1}$.  The preimage of a point on the circle gives an oriented surface with boundary called the \emph{page} of the decomposition.  Similarly, provided a compact oriented surface with non-empty boundary $\Sigma$ and a diffeomorphism $\Phi:\Sigma\rightarrow\Sigma$ which restricts to the identity on a neighborhood of $\partial\Sigma$, the \emph{abstract open book} associated to the pair $(\Sigma,\Phi)$ is the 3-manifold
\begin{equation*}
\begin{gathered}
M_{\AOB}=(\Sigma\times[0,1]) / \sim,\quad \text{where}\\
(x,1)\sim(\Phi(x),0)\,\, \forall x\in\Sigma \quad\text{and}\quad (x,\theta)\sim(x,\theta') \,\,\forall x\in\partial\Sigma;\,\,\theta,\theta'\in[0,1].
\end{gathered}
\end{equation*}
Every abstract open book $\AOB$ admits the open book decomposition $(B,\pi)=(\partial\Sigma,(x,\theta)\mapsto\theta)$ and every open book decomposition determines an abstract open book (up to conjugation of the diffeomorphism $\Phi$).  Accordingly, we will refer to either structure simply as an open book unless an isotopy class for the binding is specified and will use abstract open book notation unless otherwise specified.

Following Giroux \cite{Giroux:ContactOB}, we say that an open book $\AOB$ \emph{supports} or \emph{is compatible with}  $\Mxi$ if
\be
\item $M=M_{\AOB}$, and
\item there is a contact 1-form $\alpha$ for $\xi=\text{Ker}(\alpha)$ which is a positive length element on the binding and such that the Reeb vector field $R_{\alpha}$ is transverse to the interiors of all the pages.
\ee
In \cite{ThWi:OB} Thurston and Winkelnkemper showed that every open book decomposition $\AOB$ gives rise to a compatible contact manifold $(M_{\AOB},\xi_{\AOB})$ which depends only on $\Sigma$ and the conjugacy class of $\Phi$.  The following theorem asserts that all contact 3-manifolds arise in this way.

\begin{thm}[\cite{Giroux:ContactOB}]\label{Thm:GirCor}
Let $M$ be a closed, oriented 3-manifold.  Then,
\be
\item $\Mxi$ is supported by some $\AOB$, and
\item the mapping $\AOB\mapsto (M_{\AOB},\xi_{\AOB})$ determines a one-to-one correspondence between
    \be
    \item isotopy classes of positive contact structures on $M$ and
    \item isotopy classes of open book decompositions of $M$ up to positive stabilization.
    \ee
\ee
\end{thm}

In light of Theorem \ref{Thm:GirCor}, it is natural to ask how properties of surfaces and their diffeomorphisms translate into contact-geometric qualities.  Important progress has been made with the sobering arc criterion of Goodman \cite{Goodman:Sober} and the right-veering program of Honda-Kazez-Mati\'{c} \cite{HKM:RightVeering}.

\begin{thm} \label{Thm:GirouxStab}
Let $\Mxi$ be a contact 3-manifold.  Then the following are equivalent:
\be
\item $\xi$ is overtwisted.
\item $\Mxi$ admits a compatible open book decomposition which is a negative stabilization of some other open book decomposition for $M$.
\item $\Mxi$ has a supporting open book decomposition whose page contains a sobering arc.
\item $\Mxi$ is supported by an open book whose monodromy is not right-veering.
\ee
\end{thm}

Moreover, while every topological 3-manifold admits an open book decomposition with planar pages \cite{Alexander:OB}, not every contact 3-manifold is supported by a planar open book.

\begin{thm} [Etnyre \cite{Etnyre:Inter}] \label{Thm:Inter}
Let $\Mxi$ be a contact 3-manifold.
\be
\item  If $\Mxi$ is overtwisted, then it is supported by a planar open book decomposition.
\item  Suppose that $\Mxi$ is supported by a planar open book decomposition.  Then any symplectic filling $X$ for $\Mxi$ has connected boundary and is such that $b^{+}_{2}(X)=b^{0}_{2}(X)=0$.  If, in addition, $M$ is an integral homology sphere, then the intersection form for $X$ is diagonalizable.
\ee
\end{thm}

\begin{cor}[\cite{Etnyre:Inter}] \label{Cor:PosTB}
Suppose that $K$ is a Legendrian knot in $\Sthree$ whose Thurston-Bennequin number $tb(K)>0$. Then
\be
\item  Legendrian surgery on $K$ produces a contact manifold which cannot be supported by a planar open book, and
\item  K cannot be contained in the page of a planar open book decomposition of $\Sthree$.
\ee
\end{cor}

\begin{proof}
Suppose that $tb(K)>0$.  Legendrian surgery on $K$ gives rise to a symplectic filling $(X,\omega)$ of the contact manifold $(S^{3}_{K},\xi_{K})=\partial (X,\omega)$ obtained by surgery on $K$.  This follows from that fact that the surgery may be realized as the attachment of a symplectic 2-handle to the filling of $\Sthree$ by a 4-dimensional disk in the symplectic manifold $(\mathbb{R}^{4},dx_{1}\wedge dy_{1}+dx_{2}\wedge dy_{2})$.  See \cite{Weinstein:Handles}.  Legendrian surgery along $K$ is smoothly equivalent to a $tb(K)-1$ surgery with respect to the Seifert framing as can be seen by comparing the Seifert and contact framings of $K$ (c.f. \cite[\S 7.2]{OzbSt:SteinSurgery}).  Therefore, the union of a Seifert surface for $K$ in the filling together with the core disk of the 4-dimensional surgery 2-handle represents a non-zero class in $H_{2}(X,\mathbb{Z})$ with self intersection $tb(K)-1\geq 0$.  It follows that $b_{2}^{+}(X)=1$ so that $(S^{3}_{K},\xi_{K})$ cannot be supported by a planar open book decomposition by Theorem \ref{Thm:Inter}(2).  This establishes our first assertion.

If $K$ is contained in the page $\Sigma$ of a supporting open book decomposition of $\Sthree$, then Legendrian surgery on $K$ may be performed by precomposing the monodromy of this open book by a Dehn twist about $K$.  See Theorem \ref{Thm:TwistSurgery}.  This means that if $\Sigma$ is planar, then $(S^{3}_{K},\xi_{K})$ is supported by an open book with planar pages, contradicting the observations stated in the previous paragraph.  Therefore the second statement follows from the first.
\end{proof}

The existence result for planar open books on topological manifolds has been improved to show that for every topological link $L$ in a 3-manifold $M$, there is a planar open book decomposition of $M$ which contains $L$ in a single page.  This theorem was first proved by Calcut (see \cite[Theorem 7]{Calcut}, \cite[\S 2]{Calcut2}).  Another proof was later given by Onaran (see \cite[Theorem 1.2]{Ona:LegBook}).  This result, together with Theorem \ref{Thm:Inter} and Corollary \ref{Cor:PosTB}, imply that genus minimization for pages of open book decompositions is a purely \emph{contact-topological} problem.  These results have lead to the definition and study of the \emph{support invariants} \cite{EtOzb:Support, Ona:LegBook}.

\begin{defn}  Define the \emph{support genus, binding number}, and \emph{support norm} of $\Mxi$ by
\begin{equation*}
\begin{gathered}
sg\Mxi = min\lbrace\; g(\Sigma)\; : \; \exists \;\AOB \;\text{supporting}\; \Mxi \rbrace \\
bn\Mxi = min\lbrace\; \#(\partial\Sigma) \; : \; \exists \;\AOB \;\text{supporting} \; \Mxi \;\text{with} \;g(\Sigma)=sg\Mxi \rbrace \\
sn\Mxi = min\lbrace\; -\chi(\Sigma)\; : \; \exists \;\AOB \;\text{supporting} \;\Mxi \rbrace.
\end{gathered}
\end{equation*}
Similarly, for a Legendrian link $L$ in $\Mxi$ we may define $sg(M,\xi,L), bn(M,\xi,L)$, and $sn(M,\xi,L)$ by restricting to those open books supporting $\Mxi$ for which $L$ is contained in a single page.
\end{defn}

\begin{rmk}
In addition to Theorem \ref{Thm:Inter} there are obstructions to the existence of supporting planar open books such as those coming from Heegaard Floer homology \cite{OSS:PlanarFloer}, embedded contact homology \cite{Wendl:ECH}, symplectic fillability \cite{WendlNied:Stein}, and Dehn twist factorizations of mapping classes \cite{Wand}.  However, at the time of the writing of this paper there is no known example of a contact manifold whose support genus is greater than 1.
\end{rmk}

While ``having a common positive stabilization'' is a rather complex notion of equivalence between supporting open books,  much less is known about how contact surgery diagrams relate to one another.  Developments of this type appear in the works of Ding-Geiges \cite{DG:Surgery,DG:LWall,DG:Handles}, who have shown that every contact manifold may be obtained by contact surgery, and have listed a number of handle-slide and cancelation type moves which can be used to modify contact surgery diagrams.  Our goals in this paper are to develop Kirby moves \cite{FennRourke, Kirby} for contact manifolds by employing the open book perspective, as well as to analyze the problem of minimizing support invariants from the surgery perspective.  All results obtained are consequences of two algorithms described in Theorems \ref{Thm:Alg} and \ref{Thm:AlgLink}.

\subsection{From open books to surgery diagrams and the existence of surgery presentations}

Our first algorithm shows how to convert an open book presentation of a contact manifold to a contact surgery diagram.

\begin{thm} \label{Thm:Alg}
Suppose that the contact 3-manifold $\Mxi$ is presented as an open book ($\Sigma,\Phi$), with $\partial\Sigma$ connected and for which the monodromy $\Phi$ is given as a product of positive and negative Dehn twists on the Lickorish generators.  There is an algorithm which, from this data, creates a Legendrian link $L=L^{+}\cup L^{-}$ contained in $\Sthree$ for which contact $+1$ surgery on $L^{+}$ and contact $-1$ surgery on $L^{-}$ yields $\Mxi$.  Every connected component of $L$ is either an unknot with $tb=-1$ or an unknot with $tb=-2$.  Moreover, any two-component sublink of $L$ is an unlink, a Hopf link, or a $(-4,2)$-torus link.
\end{thm}

Section \ref{Sec:SurgBook} describes this construction in detail.

As a corollary of Theorem \ref{Thm:Alg} we obtain a new proof, and an improvement of Ding-Geiges' result in \cite{DG:LWall} which states that every contact 3-manifold may be obtained by contact $\pm1$-surgeries in $\Sthree$.  With the help of Theorem \ref{Thm:GirCor}, we present a proof which is, in spirit, exactly the same as Lickorish's elementary proof \cite{Lickorish:Surgery} that every closed, oriented topological 3-manifold admits a surgery presentation.

\begin{cor}\label{Thm:ContLickWall}
Every contact 3-manifold $\Mxi$ may be obtained by a sequence of contact $\pm 1$ surgeries on Legendrian knots in $\Sthree$.  Moreover, such a surgery presentation $L=L^{+}\cup L^{-}$ describing $\Mxi$ can be chosen to satisfy the conditions listed in Theorem \ref{Thm:Alg}.
\end{cor}

\begin{proof}
By Theorem \ref{Thm:GirCor}(1), $\Mxi$ is supported by an open book determined by some $\AOB$.  Possibly after a sequence of positive stabilizations, we may assume that the binding of $\AOB$ is connected.  We know from \cite{Lickorish:Generators} that $\Phi$ admits a factorization into a product of positive and negative Dehn twists on the curves depicted in Figure \ref{Fig:LickorishCurves}.  Now apply Theorem \ref{Thm:Alg}.
\end{proof}

\subsection{From surgery diagrams to open books and applications}

Our second algorithm provides a way of embedding a Legendrian link in $\Sthree$ into the page of a supporting open book of $\Sthree$.  Throughout the remainder of this paper, unless otherwise stated, $L$ will refer to a Legendrian link in $\Rthree$.  We write $D(L)$ for a front projection diagram of $L$.  A diagram $D(L)$ will be called \emph{non-split} if for every circle $c$ in $\mathbb{R}^{2}\setminus D(L)$, the disk $\disk^{2}\subset\mathbb{R}^{2}$ bounding $c$ contains either all or none of the components of $D(L)$.

\begin {thm} \label{Thm:AlgLink}
There is an algorithm (with choices) which assigns to each Legendrian link diagram $D(L)$, in $\mathbb{R}^{3} $ an open book $\AOB$ supporting $\Sthree$ which contains $L$ in a single page.  If $D(L)$ is non-split, the choices involved may be made in such a way that
\begin{equation} \label{Eq:Ineq}
-\chi(\Sigma) = \#(\text{Crossings of $D(L)$}) + \frac{1}{2}\cdot\#(\text{Cusps of $D(L)$}) -1.
\end{equation}
A Dehn twist factorization of the monodromy of the open book constructed is described in Theorem \ref{Thm:Mono}.
\end{thm}

\begin{figure}[h]
	\begin{overpic}[scale=.7]{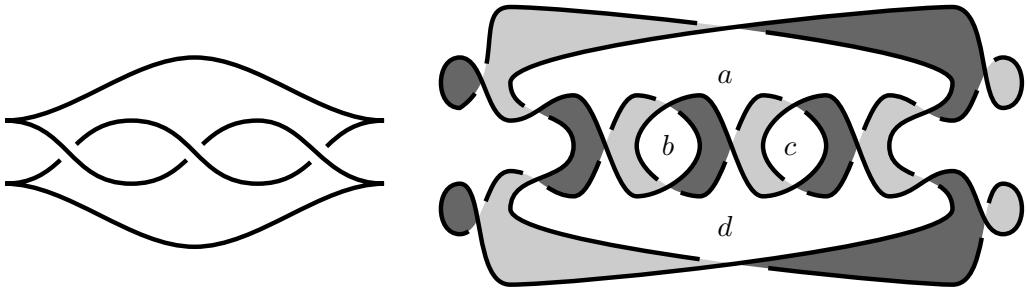}
        \put(70,20){$a$}
        \put(64.5,13){$b$}
        \put(76.5,13){$c$}
        \put(70,5){$d$}
    \end{overpic}
    \caption{On the left is a right-handed Legendrian trefoil $L\subset\Sthree$.  On the right is the page $\Sigma$ of an open book supporting $\Sthree$ which contains $L$, constructed using the algorithm of Theorem \ref{Thm:AlgLink}.  Note that $\Sigma$ has the topological type of a 3-punctured torus.  The letters $a,b,c$, and $d$ correspond to simple closed curves $\gamma_{a},\gamma_{b},\gamma_{c}$, and $\gamma_{d}$ contained in the surface.  The monodromy of the associated open book of $\Sthree$ is $\Phi=D^{+}_{\gamma_{d}}\circ D^{+}_{\gamma_{c}}\circ D^{+}_{\gamma_{b}}\circ D^{+}_{\gamma_{a}}$.  See Section \ref{Sec:Conv} and Theorem \ref{Thm:Mono} for details.  We establish the following convention which will be used throughout this paper:  When an oriented surface $\Sigma$ is drawn in the front projection of $\mathbb{R}^{3}$, the regions of $\Sigma$ on which its orientation agrees (disagrees) with the blackboard orientation will be lightly (heavily) shaded.}
    \label{Fig:Trefoil}
\end{figure}

By applying Theorem \ref{Thm:TwistSurgery}, Theorem \ref{Thm:AlgLink} can be used to convert a contact surgery diagram in $\Sthree$ into a supporting open book.

Our algorithm is far from the first of this kind to appear in the literature \cite{AkOzb:LefFib,Arikan:Genus,Plame:Algorithm,Stipsicz}, although our approach will be rather different.  All those conceived thus far have been modifications of a technique which embeds a given bridge diagram into a template convex surface, which is shown to be the page of an open book decomposition of the sphere compatible with its standard contact structure $\Sthree$.  The algorithm described in Theorem \ref{Thm:AlgLink} directly follows the proof of Theorem \ref{Thm:GirCor}(1) by describing explicit contact cell decompositions (Definition \ref{Def:ContactCell}) of $\Sthree$.  See Section \ref{Sec:GirBook} for a brief outline of the part of the proof of Theorem \ref{Thm:GirCor}(1) needed for our purposes.

Theorem \ref{Thm:AlgLink} often gives improved bounds on the support genus and norm of links in $\Sthree$.  In the case of the Legendrian trefoil knot $L$ in Figure \ref{Fig:Trefoil}, we improve the known upper bound  $sg(S^{3},\xi_{std},L) \leq 3$ \cite{Arikan:Genus} to the computation $sg(S^{3},\xi_{std},L) = 1$.  We study the support invariants of general Legendrian $(2n+1,2)$-torus knots for $n\geq 0$ in Section \ref{Sec:TorusKnots}.  While preparing this paper, the computations of Section \ref{Sec:TorusKnots} were obtained independently -- and in many cases improved using Heegaard Floer homology -- in \cite{LiWang}.

In Section \ref{Sec:PosStab} we use Theorem \ref{Thm:AlgLink} to show that overtwistedness of a contact manifold $\Mxi$ is equivalent to the existence of special type of contact surgery diagram for $\Mxi$, which corresponds to a negative stabilization in some compatible open book.  This may be viewed as Theorem \ref{Thm:GirouxStab}(2) for surgery diagrams, and complements the surgery construction of the Lutz twist as described in \cite{DGS:Lutz}.  See Section \ref{Sec:SupportInv} for these and other applications of Theorem \ref{Thm:AlgLink} to the study of support invariants.

\subsection{Mapping class relations as Kirby moves}

\begin{figure}[h]
	\begin{overpic}[scale=.7]{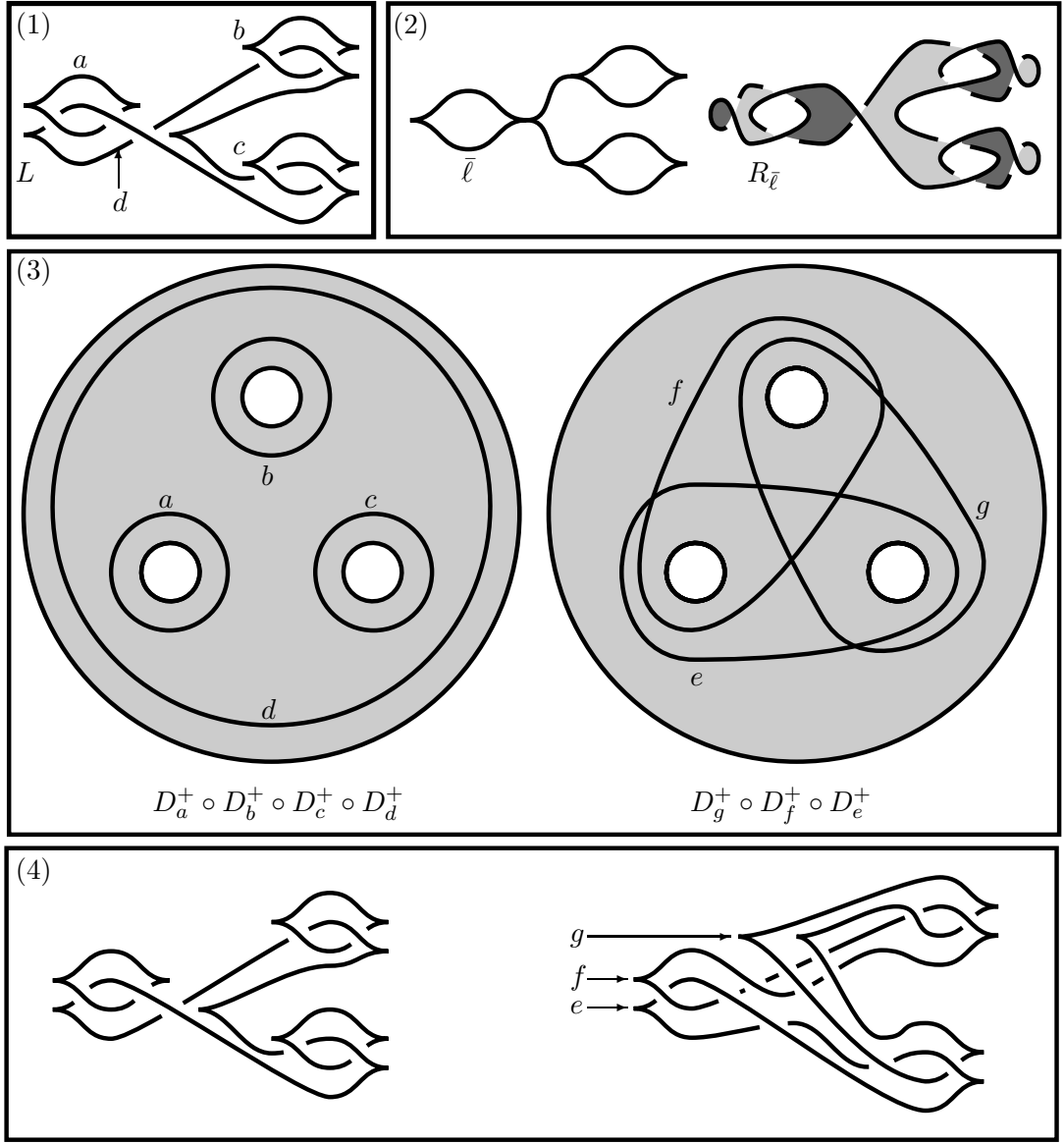}
    \put(1,97){$(1)$}
    \put(34,97){$(2)$}
    \put(1,75.5){$(3)$}
    \put(1,23){$(4)$}
    \put(1,84){$L$}
    \put(40,84){$\bar{\ell}$}
    \put(65,84){$R_{\bar{\ell}}$}
    \put(6,94){$a$}
    \put(20,96.5){$b$}
    \put(20,86.5){$c$}
    \put(9.5,81.5){$d$}
    \put(10,83.5){\vector(0,1){3.5}}
    \put(22.5,57.5){$b$}
    \put(13.5,55.5){$a$}
    \put(31.5,55.5){$c$}
    \put(22.5,37){$d$}
    \put(60,40){$e$}
    \put(58,65){$f$}
    \put(85,55){$g$}
    \put(49.5,11.25){$e$}
    \put(51,11.75){\vector(1,0){3.5}}
    \put(49.5,13.75){$f$}
    \put(51,14.25){\vector(1,0){3.5}}
    \put(49.5,17.5){$g$}
    \put(51,18){\vector(1,0){12.5}}
    \put(13,29){$D^{+}_{a}\circ D^{+}_{b}\circ D^{+}_{c}\circ D^{+}_{d}$}
    \put(60,29){$D^{+}_{g}\circ D^{+}_{f}\circ D^{+}_{e}$}
    \end{overpic}
	\caption{An example of a ribbon move.  Taking all contact surgery coefficients to be $-1$, (1) presents a contact manifold by surgery on a link $L$ of unknots.  In (2) the surgery link $\ell=a\cup b\cup c$ is embedded in a connected Legendrian graph $\bar{\ell}$, shown together with its ribbon $R_{\bar{\ell}}$.  We may regard $\ell$ as a collection of Dehn twists on $R_{\bar{\ell}}$. The curve $d$ also embeds into $R_{\bar{\ell}}$ as shown in (3).  $-1$ surgery on each component of $L$ corresponds to a diffeomorphism $D_{L}$ of $R_{\bar{\ell}}$ with Dehn twist factorization $D^{+}_{a}\circ D^{+}_{b}\circ D^{+}_{c}\circ D^{+}_{d}$.  On the right hand side of (3) is another Dehn twist factorization of $D_{L}$.  Finally in (4), Legendrian surgery presentations of the two Dehn twist factorizations are depicted.  Again all surgery coefficients are $-1$.}
    \label{Fig:Lantern}
\end{figure}

As a final application of Theorem \ref{Thm:AlgLink} we show how relations between Dehn twists in the mapping class group of a surface can be interpreted as Kirby moves (as in \cite{Kirby, FennRourke}) relating contact surgery diagrams.  Such a Kirby move, which we call a \emph{ribbon move} is executed as follows:
\be
\item Let $L$ be a contact surgery diagram in $\Sthree$ presenting the contact manifold $\Mxi$, and suppose that $\ell$ is a surgery sub-link of $L$.
\item Adjoin Legendrian arcs to $\ell$ to obtain a connected Legendrian graph $\bar{\ell}$ with ribbon $R_{\bar{\ell}}$.  The algorithm described in Theorem \ref{Thm:AlgLink} can be used to draw $R_{\bar{\ell}}$ in the front projection.
\item Each connected component of $\ell$ with its surgery coefficient correspond to a positive or negative Dehn twist on $R_{\bar{\ell}}$.  See Theorem \ref{Thm:TwistSurgery}.  Therefore $\ell$ determines an element $D_{\ell}$ of the mapping class group of $R_{\bar{\ell}}$ with a preferred Dehn twist factorization. Suppose that we can find another Dehn twist factorization
\begin{equation*}
D_{\ell} = D_{\zeta_{n}}^{\delta_{n}}\circ\dots\circ D_{\zeta_{1}}^{\delta_{1}},\quad \delta_{j}\in\lbrace +,- \rbrace
\end{equation*}
where the $\zeta_{j}$ are Legendrian realizable curves (in the sense of Section \ref{Sec:LegRel}) in $R_{\bar{\ell}}$.
\item There is a surgery link $\zeta$ contained in a neighborhood of $R_{\bar{\ell}}$ corresponding to the new Dehn twist factorization of $D_{\ell}$.  In other words, $D_{\zeta}=D_{\ell}$.  Now delete $\ell$ from $L$ and insert $\zeta$.  The surgery diagram $(L\setminus \ell)\cup\zeta$ gives an alternative surgery presentation of $\Mxi$.
\ee

The details of this construction are described in Section \ref{Sec:Kirby}.  Figure \ref{Fig:Lantern} provides an example corresponding to a lantern relation among Dehn twists on a 3-punctured disk.   The two surgery presentations in the bottom row of the figure give distinct Stein fillings of a contact structure on the Seifert fibered manifold $M(-\frac{1}{2},-\frac{1}{2},-\frac{1}{2})$.

With Theorem \ref{Thm:AlgLink} in mind, a contact surgery diagram may be thought of as an open book whose monodromy has a preferred Dehn twist factorization.  Therefore different surgery diagrams of the same contact manifold should be related by mapping class relations and Theorem \ref{Thm:GirCor}.  More precisely, in Section \ref{Sec:KirbyProof} we prove the following:

\begin{thm}\label{Thm:Kirby}
Let $\Mxi$ be a contact 3-manifold.  Suppose that $X=X^{+}\cup X^{-}$ and $Y=Y^{+}\cup Y^{-}$ are Legendrian links in $\Sthree$, both of which determine $\Mxi$ by contact surgery.  Then $X$ and $Y$ are related by a sequence of ribbon moves and Legendrian isotopies.
\end{thm}

See the next section for an explanation of the notation in the above theorem.  To the author's knowledge, the only known results regarding the modification of contact surgery diagrams are Ding-Geiges' cancelation and handle-slide moves.  In Section \ref{Sec:KirbyEx} we will reinterpret these operations as ribbon moves.  There we also provide examples of braid- and chain-relation type moves.


\section{Notation and remarks on the methods}

While we will assume that the reader is familiar with the basics of contact manifolds, open book decompositions, and contact surgery, we will quickly recall some facts needed throughout the article.  See \cite{Etnyre:OBIntro} and \cite{OzbSt:SteinSurgery} for further details.

\subsection{Conventions}\label{Sec:Conv}

$\Rthree$ will refer to the contact structure determined by the 1-form $\lambda_{std}=dz-ydx$.  $\Sthree$ denotes the contact structure on the 3-sphere considered as the boundary of the 4-disk $\disk^{4}=\lbrace ||x||\leq 1\rbrace\subset \mathbb{R}^{4}$ with Liouville 1-form $\sum (x_{j}dy_{j}-y_{j}dx_{j})$.  As the complement of a point in $\Sthree$ is contactomorphic to $\Rthree$, we will describe knots and links in $\Sthree$ by their inclusion in $\Rthree$, and will always draw knots in the \emph{front projection}, i.e. the projection to the $(x,z)$-plane in $\Rthree$.

When $\Sigma$ is an oriented surface and $\zeta$ is a simple, closed curve on $\Sigma$, a $\pm$-Dehn twist along $\zeta$ will be denoted by $D^{\pm}_{\zeta}$.  When expressing compositions of Dehn twists as a product, they will be ordered in the way that is standard for compositions of morphisms,
\begin{equation*}
\prod_{1}^{n} D^{\delta_{j}}_{\zeta_{j}} = D^{\delta_{n}}_{\zeta_{n}} \circ \cdots \circ D^{\delta_{1}}_{\zeta_{1}},\quad\delta_{j}\in\lbrace +,-\rbrace.
\end{equation*}

When performing surgery on Legendrian knots, we always express coefficients with respect to the framing given by the contact structure.

\subsection{Contact surgery}

One appropriate notion of Dehn surgery for contact 3-manifolds is \emph{contact surgery} along Legendrian knots.  Originally described in \cite{DG:Surgery}, contact surgery generalizes Weinstein's \emph{Legendrian surgery} \cite{Weinstein:Handles} for 3-manifolds.

Let $L\subset\Mxi$ be a Legendrian knot in a contact 3-manifold.  Then $L$ admits a tubular neighborhood $N(L)$ such that if we frame $L$ with $\xi$ then $\partial N(L)$ is a convex surface in $\Mxi$ with exactly 2 dividing curves of slope $\infty$.  Here we consider a contact vector field pointing out of $N(L)$, equip $\partial N(L)$ with the boundary orientation, and compute slopes by taking (``meridian'',''longitude'') as an oriented basis of $H_{1}(\partial N(L);\mathbb{Z})$ with the longitude determined by the contact framing on $L$.  The dividing curves separate $\partial N(L)$ into two annuli $\partial N(L)^{+}$ and $\partial N(L)^{-}$ which can be identified with the closures of the positive and negative regions of the convex surface $\partial N(L)$, respectively. To perform \emph{contact $\frac{1}{k}$ surgery} ($k\in \mathbb{Z}\setminus \lbrace 0 \rbrace$) along $L$, remove $N(L)$ from $M$ and then glue it back in with the identity on $\partial N(L)^{-}$ and with $-k$ (right-handed) Dehn twists along $\partial N(L)^{+}$.  It follows from the gluing theory of convex surfaces that this operation uniquely determines a contact structure on the surgered manifold.  Contact $\frac{1}{k}$ surgery on $L$ is equivalent to contact $sgn(k)$ surgery on $|k|$ copies of $L$ pushed off along the Reeb vector field of some contact 1-form for $\xi$. See \cite{DG:LWall}.

We will be primarily interested in performing contact surgery along Legendrian knots in $\Sthree$.  It is easy to check that, in this case, contact $\frac{1}{k}$ surgery is topologically a $tb(K)+\frac{1}{k}$ surgery with respect to the Seifert framing.

\begin{defn}
A \emph{contact surgery diagram} consists of a front projection diagram of a Legendrian link $L\subset\Sthree$, all of whose connected components are labeled with rational numbers of the form $\frac{1}{k}$ with $k\in\mathbb{Z}$.  In the event that every connected component of $L$ is labeled $\pm 1$ we write $L=L^{+}\cup L^{-}$ where the components of $L^{\pm}$ all have coefficient $\pm 1$.
\end{defn}

\subsection{Supporting open books}\label{Sec:GirBook}

There is a simple way to modify an open book $(\Sigma,\Phi)$ which preserves the associated contact manifold, called \emph{positive stabilization}.  This operation consists of adding a one handle to the boundary of $\Sigma$ and precomposing the monodromy with a single positive Dehn twist about a simple closed curve which intersects the co-core of the new handle exactly once.  By carrying out this procedure as in Figure \ref{Fig:Stabilize}, we can negatively and positively stabilize Legendrian knots which live in the page of such an open book.

\begin{figure}[h]
	\begin{overpic}[scale=.7]{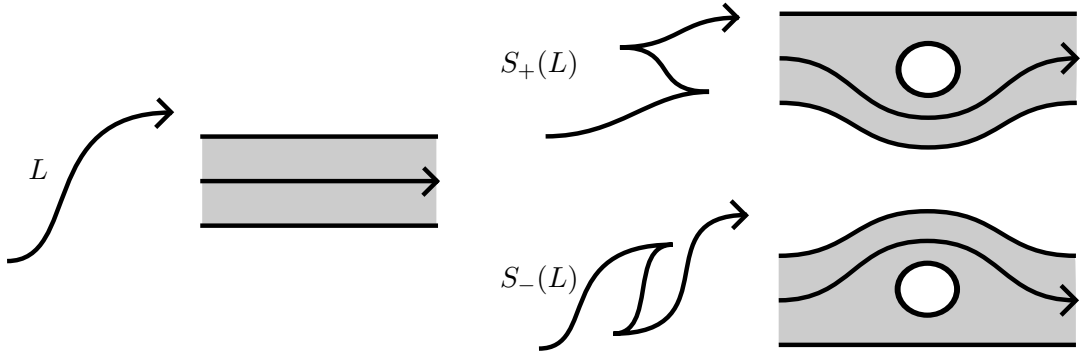}
        \put(2,16){$L$}
        \put(46,26){$S_{+}(L)$}
        \put(46,6){$S_{-}(L)$}
    \end{overpic}
	\caption{A Legendrian knot $L$, with its positive and negative stabilizations $S_{+}(L)$ and $S_{-}(L)$.  The gray surfaces represent parts of the page of an open book.  The open books on the right hand side correspond to stabilizations of the open book on the left.  The new monodromy is obtained from the old monodromy by precomposing with a positive Dehn twist about the new boundary component.}
    \label{Fig:Stabilize}
\end{figure}

The part of the proof of Theorem \ref{Thm:GirCor} which is useful for our purposes is the construction of a supporting open book from a contact manifold.  Following the exposition \cite{Etnyre:OBIntro}, we briefly outline Giroux's proof as it will guide the execution of the algorithm of Theorem \ref{Thm:AlgLink}.  The main idea is to consider cell decompositions of $M$ which have special contact geometric properties.

\begin{defn}[Giroux]\label{Def:ContactCell}
\be
\item A \emph{contact cell decomposition} of the pair $\Mxi$ is a presentation of $M$ as a cell complex such that the 1-skeleton is Legendrian, each 2-cell is convex, and each 3-cell is tight.
\item Let $L\subset\Mxi$ be Legendrian graph.  A \emph{ribbon} of $L$ is a compact, oriented surface $\Sigma\subset M$ such that
    \be
    \item $L$ is contained in $\Sigma$,
    \item there is a contact form $\alpha$ for $\Mxi$ whose Reeb vector field is everywhere positively transverse to $\Sigma$, and
    \item there is a vector field $X$ on $\Sigma$ which directs the characteristic foliation of $\Sigma$, is positively transverse to $\partial \Sigma$, and whose time-$t$ flow $\Phi^{t}_{X}$ satisfies $\cap_{(0,\infty)} \Phi^{-t}_{X}(\Sigma)=L$.
    \ee
\ee
\end{defn}

\begin{rmk}
Note if $\Sigma$ is the ribbon of a Legendrian graph $L$, then $L$ is necessarily contained in the characteristic foliation of $\Sigma$ and the boundary $\partial \Sigma$ -- equipped with the boundary orientation -- is a positive transverse link in $\Mxi$.  Provided such a surface $\Sigma$ and contact form $\alpha$ as in Definition \ref{Def:ContactCell}(2b), we can apply the flow of the Reeb vector field for $\alpha$ to find a neighborhood of $\Sigma$ of the form $[-\epsilon,\epsilon]\times \Sigma$ on which $\alpha=dz+ \alpha|_{T\Sigma}$ for some $\epsilon>0$, where $z$ is a coordinate on $[-\epsilon,\epsilon]$.  By rounding the corners of $[-\epsilon,\epsilon]\times\Sigma$ we obtain a \emph{contact handlebody} \cite{Giroux91} whose convex boundary has dividing set $\{0\}\times \Sigma$ and which naturally carries the structure of a ``half open book'' with page $\Sigma$ supporting $\xi$.  For more information, see \cite{Etnyre:OBIntro,Giroux91,OzbagciHandle}.
\end{rmk}

\begin{thm} \label{Thm:GRibbon}
Suppose that $\Mxi$ has a contact cell decomposition for which every 2-cell $\disk$ has $tb(\partial\disk,\disk)=-1$ and intersects the boundary of the ribbon $\Sigma$ of the 1-skeleton twice.  Then $\Sigma$ is the page of an open book decomposition supporting $\Mxi$.
\end{thm}

The proof -- which may be found in \cite[\S 4]{Etnyre:OBIntro} -- indicates that for every Legendrian link $L\subset M$ we can build an open book which contains $L$ in a single page by including it in the 1-skeleton of a contact cell decomposition.  Theorem \ref{Thm:AlgLink} is simply a systematic way of doing this for Legendrian links in $\Sthree$.  In the case of the unknot with $tb=-1$, Figure \ref{Fig:UnknotMovie} depicts the verification of the hypothesis of Theorem \ref{Thm:GRibbon} as well as the ``movie'' of the monodromy map.

\begin{figure}[h]
    \begin{overpic}[scale=.55, angle=-90]{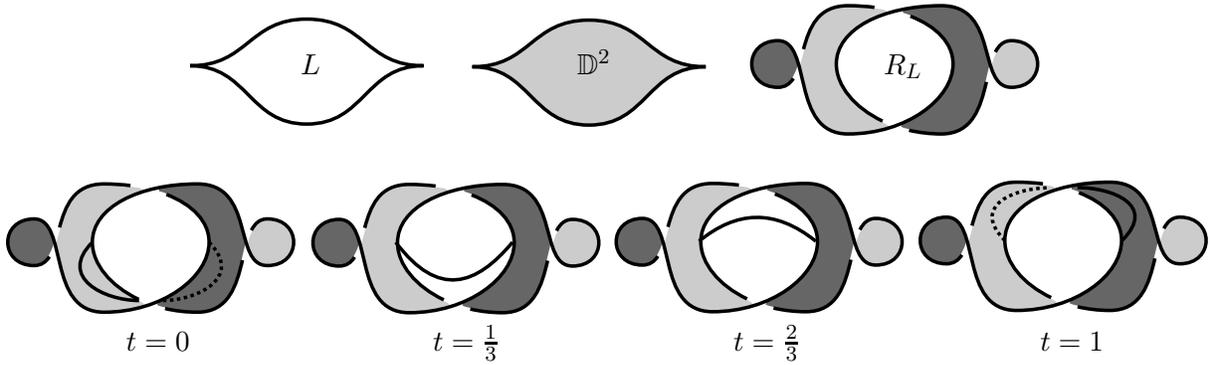}
	   \put(24.5,20){$L$}
       \put(47.5,20){$\mathbb{D}^{2}$}
       \put(73,20){$R_{L}$}
       \put(10,-3){$t=0$}
       \put(35.5,-3){$t=\frac{1}{3}$}
       \put(60.5,-3){$t=\frac{2}{3}$}
       \put(86,-3){$t=1$}
    \end{overpic}
    \vspace{5mm}
	\caption{On the upper left is a Legendrian unknot $L$ which is the 1-skeleton of a contact cell decomposition of $\Sthree$.  This cell decomposition has a single 2-cell, labeled $\mathbb{D}^{2}$ in the figure.  On the upper right is a ribbon of $L$ labeled $R_{L}$.  The bottom row shows the ``movie'' of the monodromy of the associated open book decomposition of $\Sthree$.  As the page is an annulus, we only need to analyze the image of a single arc.}
    \label{Fig:UnknotMovie}
\end{figure}

\subsection{Transverse push-offs}

Ribbons provide a natural way to associate a transverse knot to a Legendrian knot.

\begin{defn}
Let $L\subset\Mxi$ be an oriented Legendrian knot with ribbon $R_{L}$.  The \emph{positive transverse push-off} of $L$ is the component $T_{+}(L)$ of $\partial R_{L}$ on which the boundary orientation coincides with the orientation of $L$.  The \emph{negative transverse push-off} of $L$, denoted $T_{-}(L)$ is defined to be the positive transverse push-off of $-L$.
\end{defn}

\begin{figure}
	\begin{overpic}[scale=.7]{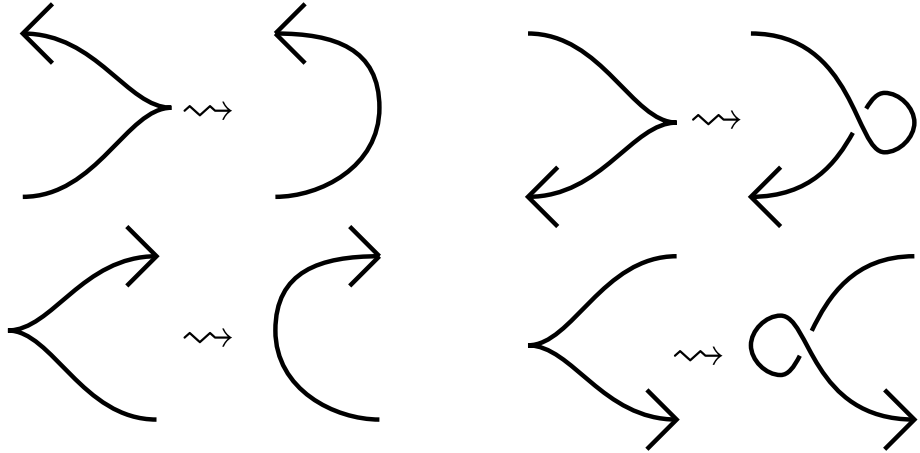}
	   \put(19,11){\huge{$\rsa$}}
       \put(73,9){\huge{$\rsa$}}
       \put(19,36){\huge{$\rsa$}}
       \put(75,35){\huge{$\rsa$}}
    \end{overpic}
	\caption{Drawing a positive transverse push-off of a Legendrian knot in the front projection near cusps.}
    \label{Fig:TransPush}
\end{figure}

Note that the positive (and negative) transverse push-offs of a Legendrian knot are positive transverse knots in $\Mxi$, isotopic to $L$ (and $-L$ respectively) as topological knots, and are uniquely determined up to transverse isotopy.  See Figure \ref{Fig:TransPush}.

\subsection{Legendrian realization, surgery curves as Dehn twists, and abuses of language}\label{Sec:LegRel}

It will be important to realize Legendrian knots on pages of open book decompositions of $\Sthree$.  Once this is achieved it is easy to construct supporting open books for contact manifolds provided contact surgery presentations and vice versa.  The essential tool in carrying out such constructions is the Legendrian realization principle \cite{Honda:TightClassI}.  We state a weak version especially suited for our purposes.

\begin{thm}
Let $C$ be a simple closed curve embedded into the interior of the ribbon $\Sigma$ of some Legendrian graph in the contact manifold $\Mxi$.  Suppose that $C$ does not bound a subsurface of $\Sigma$.  Then there is a boundary-relative isotopy of $\Sigma$ through a $[0,1]$-parameter family of surfaces $\Sigma_{t}$ such that
\be
\item the $\Sigma_{t}$ are contained in an arbitrarily small neighborhood of $\Sigma=\Sigma_{0}$, and
\item the image $C_{1}$ of $C=C_{0}$ under the isotopy is Legendrian.
\ee
\end{thm}

\begin{proof}[Sketch of the proof]
We may identify an arbitrarily small tubular neighborhood $N(\Sigma)$ of $\Sigma$ with $\Sigma\times[-\epsilon,\epsilon]$.  It is easy to construct a contact vector field transverse to the closed surface $\partial N(\Sigma) = (\Sigma\times\lbrace -\epsilon,\epsilon\rbrace) \cup (\partial\Sigma\times[-\epsilon,\epsilon])$ for which the associated dividing set is isotopic to $\partial\Sigma\times\lbrace 0\rbrace$.  Then our hypothesis guarantees that the usual Legendrian realization principle applies to $C\subset\partial N(\Sigma)$.  See \cite[Theorem 3.7]{Honda:TightClassI} for details.
\end{proof}

We will often need to find Legendrian representatives of collections of curves which are not \emph{simultaneously} Legendrian realizable.  This may be resolved by placing individual curves on different surfaces in a 1-parameter family.  For simplicity, this generalized procedure will be referred to as Legendrian realization.  For example, see Figure \ref{Fig:Lantern} in which we have Legendrian representatives of the four boundary components of a 3-punctured disk.

The preceding paragraph and above theorem indicate the necessity of another abuse of notation.  Throughout this paper, whenever we speak of the page of some open book decomposition we will be referring to its (boundary relative) isotopy class.  Thus when we say a knot is contained in some surface embedded in $\Sthree$, it is meant that the surface may be isotoped as in the statement of the above theorem so that it contains the knot.

The other principal ingredient in all of our constructions is the following statement which equates Dehn twists and contact surgeries.

\begin{thm}\label{Thm:TwistSurgery}
If $K\subset \Sigma$ is a Legendrian knot contained in the page of an open book decomposition $\AOB$ supporting the contact manifold $\Mxi$, then the manifold obtained by performing contact $\pm1$ surgery on $K$  is supported by the open book decomposition  $(\Sigma, \Phi \circ D^{\mp}_{K})$.
\end{thm}

Proofs of Theorem \ref{Thm:TwistSurgery} can be found in \cite[Theorem 5.7]{Etnyre:OBIntro} and \cite[Proposition 3]{Geiges:GeomTop}.


\section{Surgery diagrams from open books} \label{Sec:SurgBook}

\begin{figure}[h]
	\begin{overpic}[angle=-90, scale=.55]{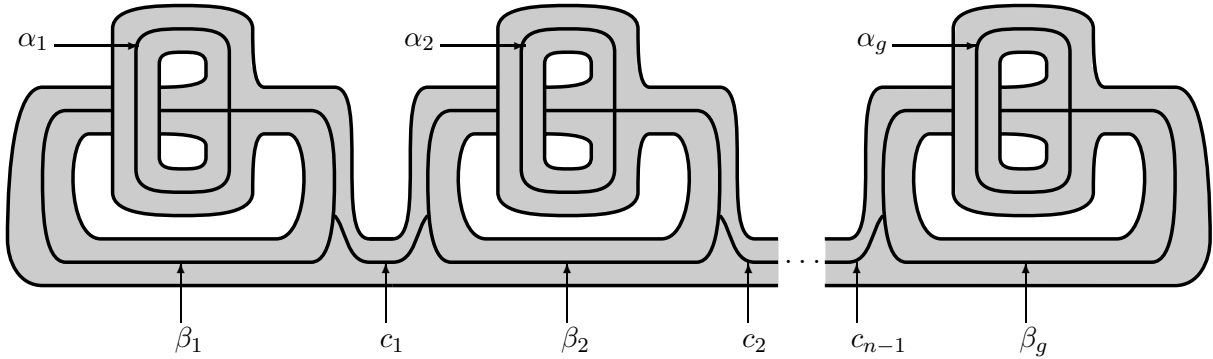}
	   \put(64.5,2){\large{$\dots$}}
       \put(4,20){\vector(1,0){7}}
            \put(1,20){$\alpha_{1}$}
       \put(36,20){\vector(1,0){7}}
            \put(33,20){$\alpha_{2}$}
       \put(73.5,20){\vector(1,0){7}}
            \put(70.5,20){$\alpha_{g}$}
       \put(14.5,-3){\vector(0,1){5}}
            \put(14,-5){$\beta_{1}$}
       \put(46.5,-3){\vector(0,1){5}}
            \put(46,-5){$\beta_{2}$}
       \put(84.5,-3){\vector(0,1){5}}
            \put(84,-5){$\beta_{g}$}
       \put(31.5,-3){\vector(0,1){5}}
            \put(31,-5){$c_{1}$}
       \put(61.5,-3){\vector(0,1){5}}
            \put(61,-5){$c_{2}$}
       \put(70.5,-3){\vector(0,1){5}}
            \put(70,-5){$c_{n-1}$}
    \end{overpic}
    \vspace{7mm}
	\caption{A compact oriented surface of genus $g$ with a single boundary component.  Define $\gamma_{j}$, for $j=1,\dots,g-1$, as follows:  Let $N_{j}$ be a tubular neighborhood of $\beta_{j}\cup c_{j}\cup\beta_{j+1}$.  Take $\gamma_{j}$ to be the boundary component of $N_{j}$ which is not homotopic to either $\beta_{j}$ or $\beta_{j+1}$.  Then Dehn twists about the $\alpha_{j},\beta_{j}$ and $\gamma_{j}$ represent Lickorish generators of the mapping class group of the surface.}
    \label{Fig:LickorishCurves}
\end{figure}

Suppose that $\Mxi$ is a contact 3-manifold supported by the open book $\AOB$, where $\Sigma$ is a genus $g$ surface with a single boundary component and the monodromy $\Phi$ is expressed as a product of positive and negative Dehn twists on the Lickorish generators described in Figure \ref{Fig:LickorishCurves}:
\begin{equation*}
\Phi=\prod_{k=1}^{n} D_{\zeta_{k}}^{\delta_{k}}\quad\text{where}\quad\delta_{k}\in\lbrace +,-\rbrace,\quad\zeta_{k}\in\lbrace \alpha_{j},\beta_{j},\gamma_{j}\rbrace.
\end{equation*}

\subsection{Algorithm 1}

The following algorithm describes how to obtain a contact surgery diagram of $\Mxi$ from this data as in Theorem \ref{Thm:Alg}.  In the next section we will show that the surgery diagram obtained presents $\Mxi$.\\

\noindent\textbf{Step 1 (Embedding $\Sigma$ in $\Sthree$).}  The curves $\alpha_{j},\beta_{j}$ and $c_{j}$ can be embedded into $\Sthree$ as described in Figure \ref{Fig:EmbedSkeleton} so that the ribbon of the graph $(\cup\alpha_{j})\cup(\cup\beta_{j})\cup(\cup c_{j})$ is diffeomorphic to $\Sigma$.  We shall henceforth consider $\Sigma$ as being contained in $\Sthree$ (or $\Rthree$).

\begin{figure}[h]
	\begin{overpic}[scale=.7]{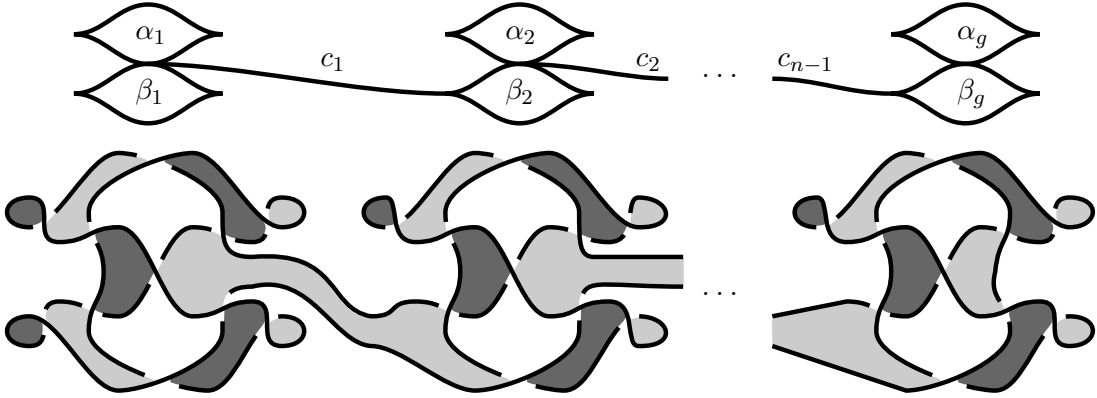}
	   \put(64,29){\large{$\dots$}}
       \put(64,9){\large{$\dots$}}
       \put(29,30){$c_{1}$}
       \put(58,30){$c_{2}$}
       \put(71,30){$c_{n-1}$}
       \put(12,32.5){$\alpha_{1}$}
       \put(46,32.5){$\alpha_{2}$}
       \put(87.5,32.5){$\alpha_{g}$}
       \put(12,27){$\beta_{1}$}
       \put(46,27){$\beta_{2}$}
       \put(87.5,27){$\beta_{g}$}
    \end{overpic}
	\caption{On top is the image of the graph $(\cup\alpha_{j})\cup(\cup\beta_{j})\cup(\cup c_{j})$.  This gives a contact cell decomposition of $\Sthree$.  A page of the associated open book for $\Sthree$ is shown on the bottom.  Here everything is drawn in the $(x,z)$-projection.}
    \label{Fig:EmbedSkeleton}
\end{figure}

After rescaling the variable $z$ on $\Rthree$ we may assume that the mapping
\begin{equation*}
[-1,1]\times\Sigma\rightarrow\Rthree,\quad (t,x)\mapsto x+(0,0,t)
\end{equation*}
is an embedding.  For $t\in [-1,1]$ and $\zeta\subset\Sigma$ we write $\zeta(t)=\zeta+(0,0,t)$.  The curves $\gamma_{j}(1)$ may be drawn in the front projection as in Figure \ref{Fig:GammaEmbed}

\begin{figure}[h]
	\begin{overpic}[scale=.7]{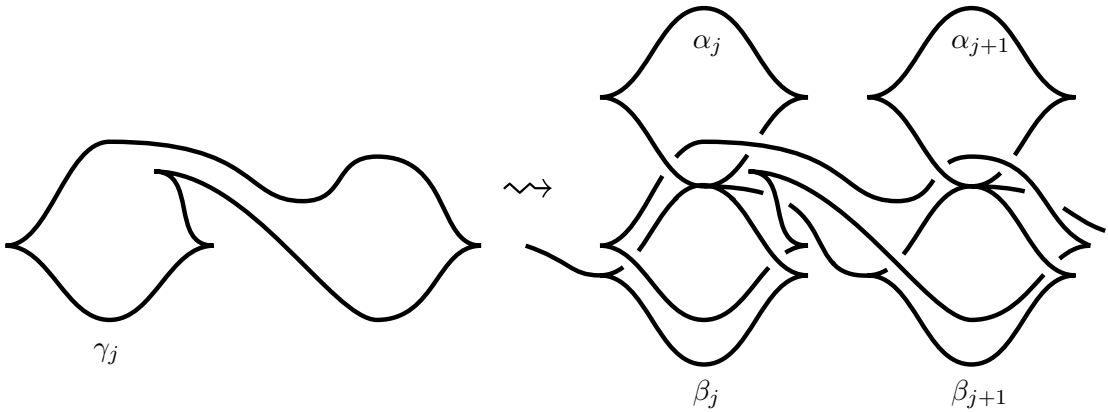}
	   \put(45,15){\huge{$\rsa$}}
       \put(8,1){$\gamma_{j}$}
       \put(62.5,-3){$\beta_{j}$}
       \put(86,-3){$\beta_{j+1}$}
       \put(62.5,29){$\alpha_{j}$}
       \put(86,29){$\alpha_{j+1}$}
    \end{overpic}
    \vspace{5mm}
	\caption{The curve $\gamma_{j}$ pushed off of the graph $(\cup\alpha_{j})\cup(\cup\beta_{j})\cup(\cup c_{j})$.}
    \label{Fig:GammaEmbed}
\end{figure}

Now we are ready to begin drawing a surgery diagram for $\Mxi$.\\

\noindent\textbf{Step 2 (Monodromy correction).}  Draw the curves $\alpha_{j}(0)$ and $\beta_{j}(-1)$ for $j=1,\dots,n$ in the front projection.  Label each curve with a surgery coefficient $+1$.\\

\noindent\textbf{Step j+2 for $j=1,\dots,n$ (Adding Dehn twists).}  Draw the curve $\zeta_{j}(j/n)$ in the diagram decorated with surgery coefficient $-\delta_{j}$.\\

\begin{figure}[h]
	\begin{overpic}[scale=.7]{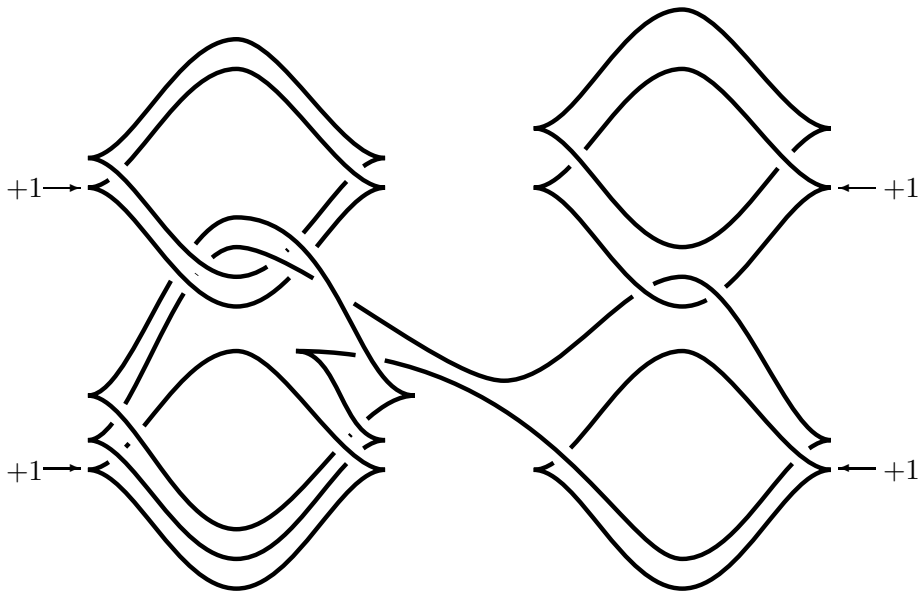}
	   \put(-6,16.5){\vector(1,0){5}}
            \put(-11,15){$+1$}
       \put(-6,54.25){\vector(1,0){5}}
            \put(-11,53){$+1$}
       \put(106,16.5){\vector(-1,0){5}}
            \put(107,53){$+1$}
       \put(106,54.25){\vector(-1,0){5}}
            \put(107,15){$+1$}
    \end{overpic}
    \vspace{5mm}
	\caption{A surgery diagram obtained from Algorithm 1.  The input is an open book whose page is a genus 2 surface with connected binding and whose monodromy is $\Phi=D^{+}_{\alpha_{2}}\circ D^{+}_{\beta_{1}}\circ D^{+}_{\gamma_{1}}\circ D^{+}_{\alpha_{1}}$.  The curves with surgery coefficient $+1$ come from Step 2 of Algorithm 1.  All other curves have surgery coefficient $-1$.}
    \label{Fig:SurgeryExample}
\end{figure}

The algorithm is now complete.  For an example of a completed diagram, see Figure \ref{Fig:SurgeryExample}.

\subsection{Justification of Algorithm 1}

 As in the previous section we consider the contact manifold $\Mxi$, presented as the open book $\AOB$ whose binding is connected.  Suppose that $\Phi$ is presented as a product of positive and negative Dehn twists on the Lickorish generators:
\begin{equation*}
\Phi = \prod_{1}^{n} D^{\delta_{j}}_{\zeta_{j}},
\end{equation*}
for some $\delta_{j}\in \lbrace +,- \rbrace$, $\zeta_{j}\in \lbrace \alpha_{i}, \beta_{i}, \gamma_{i} \rbrace$, and $n\in\mathbb{N}$.

\begin{lemma} \label{Lemma:Embed}
Let $\Sigma$ be a compact surface of genus $g$ with a single boundary component, and let $\alpha_{i}, \beta_{j}$, and $\gamma_{j}$ denote the collection of Lickorish generators as shown in Figure \ref{Fig:LickorishCurves}, respectively.  Then there is an embedding $\Psi$ of $\Sigma$ into $\Sthree$ such that
\be
\item $\Psi (\Sigma)$ is the page of an open book decomposition supporting $\Sthree$,
\item $\Psi$ sends the $\alpha_{i}$ and $\beta_{j}$ to Legendrian unknots with $tb=-1$, and the $\gamma_{j}$ to Legendrian unknots with $tb=-2$, and
\item the monodromy $\Phi_{\Psi}$ of the associated open book decomposition, $(\Psi(\Sigma), \Phi_{\Sthree})$, admits the Dehn twist factorization
\begin{equation} \label{Eq:OBMono}
\Phi_{\Psi} = \prod _{1}^{g} D_{\beta_{i}}\circ D_{\alpha_{i}}.
\end{equation}
\ee
\end{lemma}

\begin{proof}
Consider the Legendrian graph $G$ and surface $\Sigma'$ of Figure \ref{Fig:EmbedSkeleton}.  It follows from the discussion in Section \ref{Sec:Justification} that the surface $\Sigma$ in Figure \ref{Fig:EmbedSkeleton} is a ribbon of $G$, and that $\Sigma'$ is a page of an open book supporting $\Sthree$.  From the diagram it is clear that $\Sigma'$ is diffeomorphic to $\Sigma$ and that the curves $\alpha_{j},\beta_{j}$ and $c_{j}$ are embedded as desired.  Call this diffeomorphism $\Psi$.  The Equation \ref{Eq:OBMono} follows from Theorem \ref{Thm:Mono}.
\end{proof}

\begin{prop}
In the notation of Step 1 of Algorithm 1, $\Mxi$ is equal to the contact manifold obtained by contact Dehn surgery on the link $L=L^{+}\cup L^{-}$ where
\begin{equation} \label{Eq:Surgery}
\begin{gathered}
L^{+} = \left( \bigcup _{\delta_{j}=-} \Psi(\zeta_{j})(j/n)\right)\cup\left( \bigcup_{1}^{g} \Psi(\beta_{i})(-1)\right)\cup\left(\bigcup_{1}^{g} \Psi(\alpha_{i})(0)\right) \\
\text{and}\quad L^{-} = \bigcup _{\delta_{j}=+} \Psi(\zeta_{j})(j/n).
\end{gathered}
\end{equation}
\end{prop}

\begin{proof}
We begin by considering the open book decomposition $(\Psi(\Sigma), \Phi_{\Sthree})$ and the monodromy factorization given by Equation \ref{Eq:OBMono}.  Note that by our construction, all of the Lickorish curves are now realized as Legendrian knots on distinct pages of this open book decomposition.  Applying contact $+1$ surgeries along the $\beta_{i}$ on $\Psi(\Sigma)_{-1}$ and along the $\alpha_{i}$ on $\Psi(\Sigma)_{0}$, we obtain an open book decomposition whose page is $\Sigma$ and whose monodromy is the identity map as the collection of these surgeries cancels the Dehn twists from Equation \ref{Eq:OBMono}.  Applying the remaining surgeries in the fashion specified by the statement of the theorem will recreate $\AOB$ which supports $\Mxi$ by Theorem \ref{Thm:TwistSurgery}.
\end{proof}

Now the proof of Theorem \ref{Thm:Alg} is complete.  If $\Mxi$ admits a compatible open book decomposition modeled on a torus with a single disk removed, we can do better:

\begin{cor}\label{Cor:PuncturedTorus}
Let $\Mxi$ be a contact 3-manifold which admits a supporting open book $\AOB$ for which $\Sigma$ is a torus with a single disk removed.  Then $\Mxi$ is obtained by contact surgery on a link of Legendrian unknots, all of which have $tb=-1$ and such that every two component sub-link is either a Hopf link or an unlink.
\end{cor}

This is immediate from the proof of Theorem \ref{Thm:Alg} and the fact that the $\gamma$ Lickorish curves are non-existent in this case.

\begin{rmk}\label{Rmk:ChernClass}
Due to Gompf's formula relating Chern classes of contact structures and rotation numbers of Legendrian knots \cite[Proposition 2.3]{Gompf} (see also \cite[\S 11.3]{OzbSt:SteinSurgery}), Corollary \ref{Cor:PuncturedTorus} extends \cite[Lemma 6.1]{EtOzb:Support}, which states that the Chern class of any contact structure supported by an open book whose page is a punctured torus must be zero.
\end{rmk}


\section{Open books from surgery diagrams} \label{algorithm}

In this section we show how to embed a Legendrian link into the page of an open book decomposition supporting $\Sthree$ as in Theorem \ref{Thm:AlgLink}.  Let $L$ be a Legendrian link in $\Sthree$ with front diagram $D(L)$.  In Section \ref{Sec:Alg} we present the part of the algorithm which builds the page of an open book decomposition from $D(L)$.  In Section \ref{Sec:Justification} we prove that the surface obtained in Section \ref{Sec:Alg} is indeed the page of an open book supporting $\Sthree$.  A Dehn twist factorization of the monodromy of this open book is described in Section \ref{Sec:Mono}.  In Section \ref{Sec:TakingControl} we show how the Euler characteristic and number of boundary components of the surface constructed can be controlled.

\subsection{Algorithm 2}\label{Sec:Alg}

For simplicity, we assume that each cusp of $D(L)$ is tangent to a line of the form $\mathbb{R}\times\lbrace z_{0} \rbrace$ in the $(x,z)$-plane.  Any Legendrian front diagram may be slightly perturbed so that this condition holds.\\

\noindent\textbf{Step 1 (Crossing completion).} For every crossing of $D(L)$ we adjoin to $L$ a Legendrian arc of the form $\gamma(t) = (x_{0},y_{0}+t,z_{0})$ connecting the two points on $L$ associated to the crossing.  The completed Legendrian graph, $L'$ can be pictorially represented by completing every crossing to an ``X'' as depicted in Figure \ref{Fig:CrossingCompletion}.  We denote the completed diagram by $D(L')$.\\

\begin{figure}[h]
	\begin{overpic}[scale=.7]{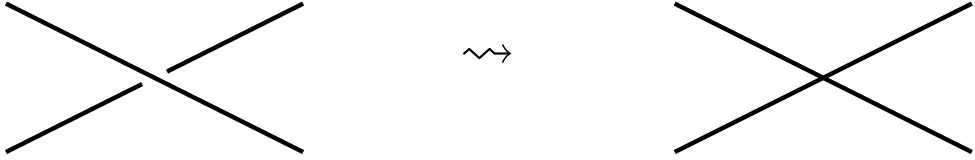}
	   \put(47,9){\huge{$\rsa$}}
    \end{overpic}
	\caption{Step 1 of the algorithm.}
    \label{Fig:CrossingCompletion}
\end{figure}

\noindent\textbf{Step 2 (Reduction to the connected case).} If $D(L')$ is connected, relabel $L'$ as $L''$ and continue to Step 3.  Otherwise, label the connected components of $D(L)$ as $D(L')_{j};\; j=1,\dots,n$. Choose $n-1$ Legendrian arcs in $\Rthree$, $\gamma_{1},\dots,\gamma_{n-1}$ such that the front diagram $D(\cup_{j}\gamma_{j})$ has embedded interior (i.e. has no crossings), and such that
\be
\item the interior is disjoint from $D(L')$, and
\item $\partial D(\gamma_{j})$ consists of a cusp on $D(L')_{j}$ and another on $D(L')_{j-1}$.
\ee
We also require that near each point of $\partial D(\cup_{j}\gamma_{j})$, the diagram is tangent to a line of the form $\gamma(t) = (x_{0},y_{0}+t,z_{0})$.

Let $L''=L'\cup(\cup_{j} \gamma_{j})$ with front diagram $D(L'')$.  An example is depicted in Figure \ref{Fig:LDoublePrime}.

\begin{figure}[h]
	\begin{overpic}[scale=.7]{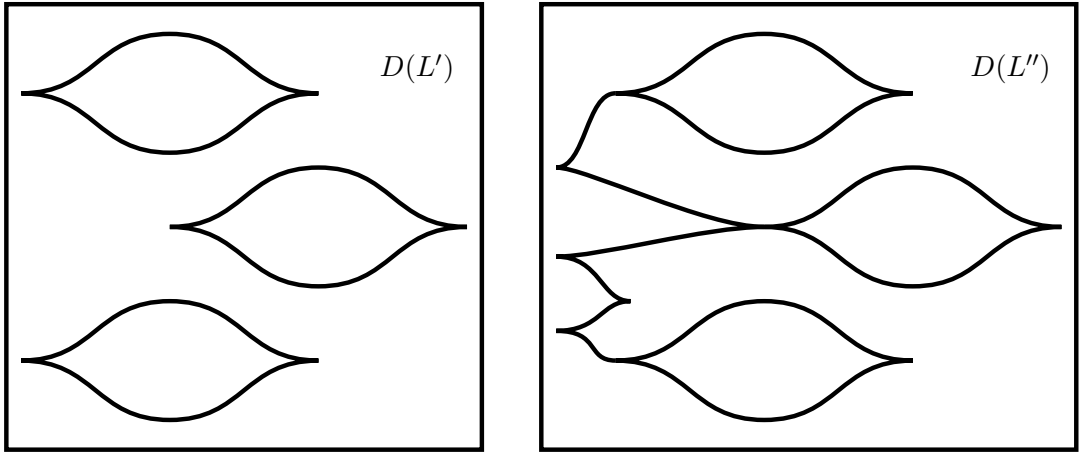}
	   \put(35,35){$D(L')$}
       \put(90,35){$D(L'')$}
    \end{overpic}
	\caption{Adjoining Legendrian arcs to $L'$, ensuring connectedness of $L''$.}
    \label{Fig:LDoublePrime}
\end{figure}

\noindent\textbf{Step 3 (Partitioning unknots)}  Now every embedded disk $\disk$ in $\mathbb{R}^{2}\setminus D(L'')$ with boundary on $D(L'')$ lifts to a disk in $\Rthree$ bounding a piecewise smooth Legendrian unknot in $L''$.  If, after smoothing, $\partial\disk$ has Thurston-Bennequin invariant less than $-1$, partition $\disk$ into disks with $tb=-1$ by adjoining non-destabilizeable Legendrian arcs to $L'$, both of whose boundary points live on the cusps of $\partial \disk$.  Do this in such a way that the union of the new arcs has embedded interior.  Note that there is in general no unique way to choose the cusps along which the arcs will be connected.  An example is depicted in Figure \ref{Fig:DiskPartition}.  Denote the Legendrian graph obtained by $\bar{L}$.  We will later see that $\bar{L}$ is the 1-skeleton of a contact cell decomposition of $\Sthree$.\\

\begin{figure}[h]
	\begin{overpic}[scale=.7]{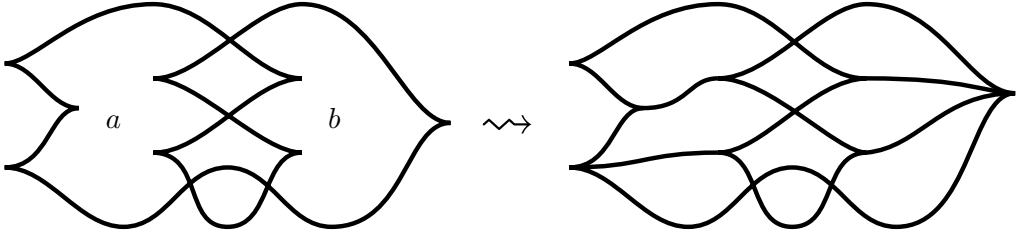}
	   \put(47,9){\huge{$\rsa$}}
        \put(10,10){$a$}
        \put(32,10){$b$}
    \end{overpic}
	\caption{A typical execution of Step 3.  Here only the disks labeled $a$ and $b$ need to be cut into smaller pieces.}
    \label{Fig:DiskPartition}
\end{figure}

\noindent\textbf{Step 4 (The ribbon near cusps).}  Now develop the ribbon of $\bar{L}$ near its cusps in the front projection.  We show how this is done in Figure \ref{Fig:CuspRibbon}.  In a neighborhood of a given cusp, $\bar{L}$ may be described by a Legendrian arc $\ell$ which we may parameterized so that $\frac{\partial z}{\partial t} \geq 0$.  The boundary of the ribbon near $\ell$ will consist of one positive and one negative transverse push-off of $\ell$.\\

\begin{figure}[h]
	\begin{overpic}[scale=.7]{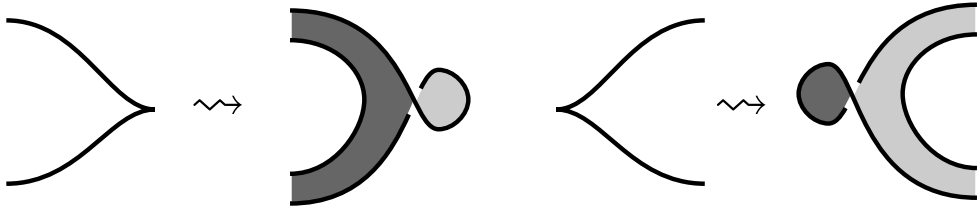}
	   \put(19,9){\huge{$\rsa$}}
       \put(73,9){\huge{$\rsa$}}
    \end{overpic}
	\caption{A picture of the ribbon of $\bar{L}$ near the cusps.  The orientation of the surface inherited from $\xi$ and the blackboard orientation agree on the lightly shaded regions and disagree on the heavily shaded regions.}
    \label{Fig:CuspRibbon}
\end{figure}

\noindent\textbf{Step 5 (The ribbon near singularities).}  Now we draw the ribbon of $\bar{L}$ in the front projection in neighborhoods of the singularities of $D(\bar{L})$.  Let $p\in D(\bar{L})$ be a singularity of $D(\bar{L})$.  In an arbitrarily small neighborhood of $p\in\mathbb{R}^{2}$ there may be an arbitrarily large number of Legendrian arcs in $D(\bar{L})$ emanating from $p$.  Suppose that there are $m$ arcs $\alpha_{1},\dots,\alpha_{m}$ in this neighborhood which lie to the left of $p$ and $n$ arcs $\beta_{1},\dots,\beta_{n}$ which lie to the right of $p$.  Suppose all of the $\alpha_{j}$ and $\beta_{j}$ are oriented so that they point out of $p$ and are indexed so that $\alpha_{j+1}$ lies above $\alpha_{j}$ and $\beta_{j+1}$ lies above $\beta_{j}$ as shown in Figure \ref{Fig:SingRibbon}. Then the boundary of the ribbon of $\bar{L}$ near $p$ is
\be
\item $T_{+}(\alpha_{j+1}\cup(-\alpha_{j}))$ for $j=1,\dots,m-1$,
\item $T_{+}(\beta_{j+1}\cup(-\beta_{j}))$ for $j=1,\dots,n-1$,
\item $T_{+}(\beta_{1}\cup(-\beta_{n}))$ if $m=0$; $T_{+}(\beta_{1}\cup(-\alpha_{m}))$ if $m\neq 0$, and
\item $T_{+}(\alpha_{1}\cup(-\alpha_{n})$ if $n=0$;  $T_{+}(\alpha_{1}\cup(-\beta_{n})$ if $n\neq 0$.\\
\ee

\begin{figure}[h]
	\begin{overpic}[scale=.7]{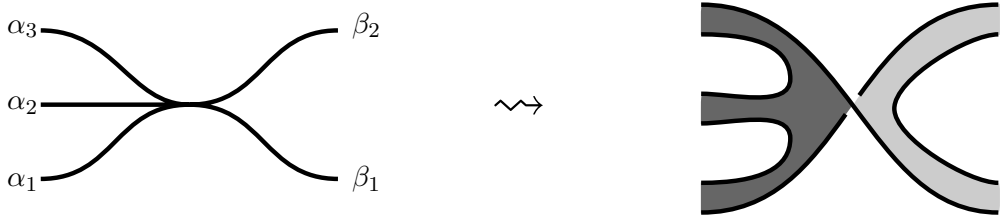}
	   \put(47,10){\huge{$\rsa$}}
        \put(-3.5,19){$\alpha_{3}$}
        \put(-3.5,11){$\alpha_{2}$}
        \put(-3.5,3){$\alpha_{1}$}
        \put(32.5,19){$\beta_{2}$}
        \put(32.5,3){$\beta_{1}$}
    \end{overpic}
	\caption{A picture of the ribbon of $\bar{L}$ near a typical singularity.  The arcs $\alpha_{j}$ and $\beta_{j}$ are labeled as in Step 5.}
    \label{Fig:SingRibbon}
\end{figure}

\noindent\textbf{Step 6 (The ribbon along the remaining non-destabilizeable arcs).}  The next step consists of completing the ribbon for $\bar{L}$ by adjoining strips along those subarcs which are free of cusps and crossings.  There is a unique (up to isotopy) way to do this due to the required transversality of the ribbon with the vector field $\partial_{z}$.  Suppose that $\ell$ is such an arc from $p\in\mathbb{R}^{2}$ to $q\in\mathbb{R}^{2}$, oriented left to right.  The previous steps in the construction have forced the blackboard orientation of the ribbon to disagree with the $\partial_{z}$-orientation at $p$ and to agree with the $\partial_{z}$ orientation at $q$.  Therefore we can complete the ribbon along $\ell$ as in Figure \ref{Fig:ArcRibbon}.

\begin{figure}[h]
	\begin{overpic}[scale=.7]{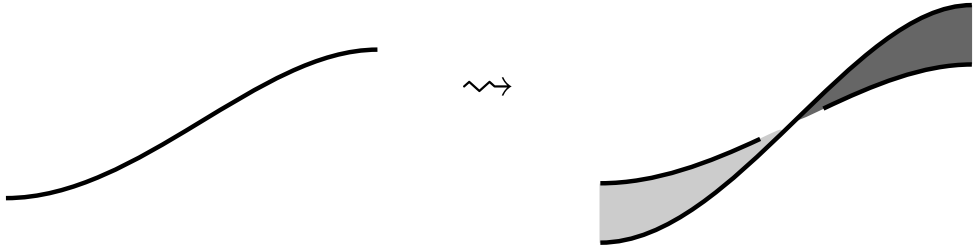}
	    \put(47,15){\huge{$\rsa$}}
    \end{overpic}
	\caption{A picture of the ribbon of $\bar{L}$ along a non-destabilizeable arc as described in Step 6.}
    \label{Fig:ArcRibbon}
\end{figure}

We write $\Sigma$ for the surface constructed after the completion of Step 6.  In the next section we will show that $\Sigma$ is the page of an open book supporting $\Sthree$.  For an easy example, see Figure \ref{Fig:Trefoil}.  In this example the second and third steps of the algorithm are trivial.

\begin{rmk}
The front projection of the surface $\Sigma$ can sometimes be simplified by Reidemeister-type moves.  See Figure \ref{Fig:Reidemeister} for an example.
\end{rmk}

\begin{figure}[h]
    \begin{overpic}[scale=.7]{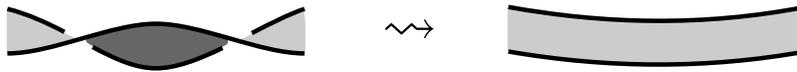}
        \put(47,3.5){\huge{$\rsa$}}
    \end{overpic}
    \caption{Applying a Reidemeister move to the surface $\Sigma$.}
    \label{Fig:Reidemeister}
\end{figure}

\subsection{Justification of Algorithm 2}\label{Sec:Justification}

The purpose of this section is to prove Theorem \ref{Thm:AlgLink} using the algorithm described in the previous section.  Before beginning the proof we introduce some vocabulary which will simplify the discussion.

\begin{defn}
Let $D(\bar{L})$ be a front diagram for a Legendrian graph $\bar{L}\subset\Rthree$.  An embedded disk $\disk\subset\mathbb{R}$ is called \emph{small} if Int($\disk)\cap D(\bar{L})=\emptyset$ and $\partial\disk\subset D(\bar{L})$.
\end{defn}

Note that for every small disk in $\mathbb{R}^{2}$, there is a unique (up to boundary relative isotopy) lift to an embedded disk in $\Sthree$ with boundary in $\bar{L}$.  Then there is no ambiguity in calling such lifted disks in $\Sthree$ \emph{small}.  Small disks will be our candidates for the 2-cells of a contact cell decomposition of the 3-sphere which contains $L$ in its 1-skeleton.

Although small disks will generally have piecewise smooth boundary they can be approximated by smooth disks.  Let $\Sigma$ be a ribbon for the Legendrian graph $\bar{L}$.  If $\disk$ is a small disk in $\Rthree$, then $\partial\disk\subset\Sigma$ can be isotoped to a smooth simple, closed curve in $\Sigma$ by the Legendrian realization principle.

\begin{prop}\label{Prop:TBCompute}
Let $\disk\subset\Rthree$ be a small disk for $D(\bar{L})$.  Then after smoothing $\disk$ to a smooth disk $\disk'$,
\begin{equation*}
tb(\partial \disk')=\#(\text{Left and right cusps of $D(\partial\disk)$}).
\end{equation*}
\end{prop}

\begin{proof}  Any cusps which are not standard left or right cusps will be eliminated after smoothing as described in the paragraph preceding Proposition \ref{Prop:TBCompute}.  Now apply the usual formula (c.f. \cite[Section 2.6.2]{Etnyre:KnotNotes}) used to compute the Thurston-Bennequin number of a smooth Legendrian knot:
\begin{equation} \label{Eq:Tb}
tb(K)= \#(\text{Positive crossings of $D(K)$})-\#(\text{Negative crossings of $D(K)$})-\frac{1}{2}\cdot\#(\text{Cusps of $D(K)$}).
\end{equation}
\end{proof}

\begin{defn}
For simplicity, if $\disk$ is a small disk in $\Rthree$ the \emph{Thurston-Bennequin number of $\disk$} will refer to the usual Thurston-Bennequin number of the boundary of a smooth disk with Legendrian boundary approximating $\disk$.  A small disk $\disk$ will be called \emph{elementary} if $tb(\disk)=-1$  The boundary of an elementary disk will be referred to as an \emph{elementary cycle}.
\end{defn}

\begin{prop}
In the notation of Section \ref{Sec:Alg}, the Legendrian graph $\bar{L}$ is the 1-skeleton of a contact cell decomposition of $\Sthree$ whose ribbon is $\Sigma$.  It follows that $\Sigma$ is the page of an open book decomposition of $\Sthree$ which contains the Legendrian link $L$.
\end{prop}

\begin{proof}
We will follow the algorithm step by step, in the end ensuring that the hypothesis of Theorem \ref{Thm:GRibbon} holds for $\bar{L}$ and $\Sigma$.

As we will also only be considering Legendrian arc segments we adopt the following conventions for local orientations.  For arcs which contain no cusps or crossings we always assume they are oriented ``from left to right'' so that $\frac{dx}{ds} > 0$.  Any arc which contains a single cusp is assumed to be oriented in such a way that $\frac{dz}{ds} > 0$.\\

\noindent\textbf{Steps 1-3.}
Steps 1 through 3 consist of connecting arcs to $L$.  Each arc is drawn in the front projection, and we must make sure that the resultant diagram $D(\bar{L})$ lifts to a Legendrian graph in $\Rthree$ which is homeomorphic to the graph $D(\bar{L})$.  In the case of Step 1 this is trivial as the added arcs are described by explicit parametrization.

For Step 2 this is established by requiring that each cusp of $L$ be tangent to a line of the form $\lbrace z= z_{0} \rbrace\subset \mathbb{R}^{2}$ for some $z_{0}\in\mathbb{R}$.  If an arc added at Step 2 also satisfies this tangency condition and is otherwise embedded in $\mathbb{R}^{2}-D(L)$, then it will have a lift to a Legendrian arc whose interior is disjoint from $L$, and whose endpoints lie on $L$ at points of the form $(x_{0},0,z_{0})$.

An endpoint of an arc $\ell$ added in Step 3 will end on a cusp or vertex of the graph $D(L'')$.  A cusp endpoint or an endpoint living on a vertex created in Step 2 may be justified as in the above paragraph.  When an endpoint lands on a vertex of $D(L'')$ created in Step 1, then again we require that $D(\ell)$ is tangent to a line of the form $\lbrace z= z_{0} \rbrace\subset \mathbb{R}^{2}$ near the endpoint.  If $\gamma$ is the arc of the form $t\mapsto(x_{0},t,z_{0})$ then $\ell$ will touch $\gamma$ at the single point $(x_{0},0,z_{0})$.\\

\noindent\textbf{Step 4.} The ribbon of $\bar{L}$ near a cusp can be described in the front projection by Figure \ref{Fig:CuspRibbon}.  Note that, away from the singular points in the diagram, each piece of the ribbon inherits two local orientations: one given by the Reeb vector field $\partial_{z}$, and the other given by ``the blackboard'' $\partial_{y}$.  We have indicated where these local orientations agree and disagree by shading the ribbon:  It is heavily shaded where these orientations disagree and lightly shaded where there orientations agree.

Near a cusp, the boundary of the ribbon consists of one positive and one negative transverse push-off.  For a right pointing cusp, note that by our orientation conventions $\frac{dy}{ds} <0$ and $\frac{dz}{ds} \geq 0$ with equality only at the point for which $y\circ\gamma = 0$.  Thus we can arrange that the transverse push-offs and surface are as in the figure.  The argument for a left pointing cusp is exactly the same except that in this case we have $\frac{dy}{ds} > 0$ and so this piece of the ribbon comes with a lighter shading.\\

\noindent\textbf{Step 5.}  Consider the $(x,y)$-projection of $\bar{L}$ near a singularity.  From this point of view, it is clear that if $\partial_{z}$ is to be transverse to the ribbon of $\bar{L}$ then the boundary of the ribbon must consist of the push-offs described in Step 5 of the algorithm.\\

\noindent\textbf{Step 6.}  Note that we can require the ribbon of $\bar{L}$ to be everywhere transverse to $\partial_{z}$.  Similarly, any surface which retracts onto $\bar{L}$ and satisfied this transversality condition on its interior and has transverse boundary is a ribbon of $\bar{L}$.  The surface shown in Figure \ref{Fig:ArcRibbon} satisfies these conditions.\\

\noindent\textbf{Applying Theorem \ref{Thm:GRibbon}.} Now we prove that $\Sigma$ of $\bar{L}$ the page of some open book decomposition for $\Sthree$, having already established that it is a ribbon of $\bar{L}$.  In order to verify this fact it suffices to check that $\Sigma$ fits the hypotheses of Theorem \ref{Thm:GRibbon}.

The Legendrian graph $\bar{L}$ is the 1-skeleton of cell decomposition of the 3-sphere whose 2-cells are the elementary disks.  Each 2-cell is (possibly after a $\mathcal{C}^{\infty}$-small perturbation) convex with Thurston-Bennequin number $-1$.  The union of $\bar{L}$ with all of the 2-cells is simply a piecewise smooth disk, whose complement is homeomorphic to a 3-ball which is the only 3-cell in this cell decomposition.  As $\Sthree$ is tight, the restriction of the contact structure to the 3-cell is also tight.

Now we must ensure that all of the interiors of the elementary disks are disjoint and that the interior of each disk intersects $\partial\Sigma$ exactly twice.  The fact that all of the interiors of the elementary disks are disjoint follows from the fact that the union of their interiors are disjointly embedded into the front projection diagram.  To guarantee that each elementary disk $\disk$ intersects the boundary of the ribbon exactly twice, first note that by our construction and the fact that $tb(\partial\disk)=-1$, $\Sigma$ is positive (is lightly shaded) near the left cusp and negative (is heavily shaded) near the right cusp.  Moreover, by our construction, the ribbon along each elementary cycle $\partial\disk$ will be the same as in the case of the Legendrian unknot with $tb=-1$ in $\Sthree$ with plumbings of some extraneous bands away from the left and right pointing cusps of $\partial\disk$ in such a way that the bands are always transverse to the vector field $\partial_{z}$.  As the transversality and geometric intersection numbers of $\disk$ with the boundary of the ribbon are invariant under plumbing bands in this way we see that the desired properties hold.

Finally, we have the ribbon of a contact cell decomposition of $\Sthree$ which intersects every 2-cell exactly twice.  Therefore, by Theorem \ref{Thm:GRibbon} there is some diffeomorphism $\Phi$ of $\Sigma$ so that $(\Sigma,\Phi)$ is equivalent to $\Sthree$ via Theorem \ref{Thm:GirCor}, and contains $L$ in a single page.
\end{proof}

\subsection{A Dehn twist presentation of the monodromy}\label{Sec:Mono}

In this section we find a Dehn twist factorization of the monodromy $\Phi$ of the open book $\AOB$ supporting $\Sthree$, where $\Sigma$ is the surface constructed in Algorithm 2.

\begin{defn}  Let $\Sigma$ be a compact oriented surface with $\partial\Sigma\ne\emptyset$.  An \emph{arc basis} of $\Sigma$ is a collection $\lbrace \gamma_{j} \rbrace_{1}^{n}$ of disjoint embedded arcs in $\Sigma$ such that
\be
\item $\partial \gamma_{j}\subset\partial\Sigma$ for all $j=1,\dots,n$ and
\item $\Sigma\setminus\cup\gamma_{j}$ is homeomorphic to a disk.
\ee
\end{defn}

Note that if $\lbrace \gamma_{j}\rbrace$ is an arc basis of $\Sigma$, then the isotopy class of the diffeomorphism $\Phi$ is determined by the isotopy classes of the images $\Phi(\gamma_{j})$.  To state the monodromy factorization theorem we need one more vocabulary item.

\begin{defn}
In the notation of Algorithm 2, let $\disk$ and $\disk'$ be distinct elementary disks for the Legendrian graph $\bar{L}$ with boundaries $C$ and $C'$, respectively.  If $\emptyset\ne C\cap C'\subset \bar{L}$ we write $\disk>\disk'$ if $\partial_{z}$ points into $\disk$ along $C\cap C'$.
\end{defn}

Note that in the above definition, it is impossible that both $\disk >\disk'$ and $\disk'>\disk$ as otherwise $C$ or $C'$ would be destabilizeable.  This would violate the definition of an elementary disk.

\begin{thm}\label{Thm:Mono}
The monodromy $\Phi$ of the open book of $\Sthree$ with page $\Sigma$ constructed in Algorithm 2 can be written as
\begin{equation*}
\Phi = D^{+}_{C_{1}}\circ \dots \circ D^{+}_{C_{n}}.
\end{equation*}
Here $\lbrace \disk_{j} \rbrace_{1}^{n}$ is the collection of elementary disks with $C_{j}=\partial\disk_{j}$, indexed so that $\disk_{j}>\disk_{i}$ implies $j>i$.
\end{thm}

\begin{proof}
Suppose that the elementary disks are indexed as in the statement of the theorem.  We will show that
\be
\item the $\disk_{j}$ naturally determine an arc basis $\lbrace \gamma_{j} \rbrace$ of $\Sigma$ and the images $\Phi(\gamma_{j})$,
\item $\Phi(\gamma_{j})=D^{+}_{C_{j}}(\gamma_{j})$, and
\item $D^{+}_{C_{1}}\circ\dots\circ D^{+}_{C_{n}} = D^{+}_{C_{j}}(\gamma_{j})$.
\ee
This suffices to prove the theorem.\\

\noindent\textbf{Step 1.}  Let $N(\Sigma)$ be an arbitrarily small tubular neighborhood of $\Sigma$.  The proof of Theorem \ref{Thm:GRibbon}(1) (see \cite{Etnyre:OBIntro}) indicates that
\be
\item $N(\Sigma)':=S^{3}\setminus N(\Sigma) $ is contactomorphic to $N(\Sigma)$ and
\item $N(\Sigma)'\setminus (N(\Sigma)'\cap(\cup \disk_{j}))$ is contactomorphic to a standard 3-ball $(B^{3},\xi_{std})$.
\ee
We can write $\Sthree$ as the open book $(M_{\AOB},\xi_{\AOB})$ where
\begin{equation*}
\begin{gathered}
N(\Sigma) = [0,1/2]\times\Sigma/((t,x)\sim(t',x)\; \forall \;x\in \partial\Sigma),\quad\text{and} \\
N(\Sigma)'= [1/2,1]\times\Sigma/((t,x)\sim(t',x)\; \forall \;x\in \partial\Sigma)
\end{gathered}
\end{equation*}
in the notation of the introduction.

The elementary disks determine an arc basis of $\Sigma$ in the following way:  Consider $\disk_{j}' = \disk_{j}\cap N(\Sigma)'$.  The boundary $\partial N(\Sigma)$ of $N(\Sigma)$ is a convex surface with
\begin{equation*}
\begin{gathered}
(\partial N(\Sigma))^{+} = \lbrace 1/2 \rbrace \times \Sigma /\sim, \quad (\partial N(\Sigma))^{-} = \lbrace 0 \rbrace \times \Sigma /\sim,\\
\text{and dividing set} \quad\partial \Sigma.
\end{gathered}
\end{equation*}
Then $\gamma_{j}:=(\partial \disk_{j}')\cap (N(\Sigma))^{+}$ is an arc basis of $(\partial N(\Sigma))^{+}$.  As $(\partial N(\Sigma))^{+}$ is isotopic in $N(\Sigma)$ to $\Sigma$, this gives an arc basis of $\Sigma$ which we again denote by $\lbrace \gamma_{j}\rbrace$.  Similarly, another arc basis of $\Sigma$ is determined by $(\partial\disk'_{j})\cap(\partial N(\Sigma))^{-}$ which gives $\Phi(\gamma_{j})$.\\

\noindent\textbf{Step 2.}  We now explicitly describe the arcs $\gamma_{j}$ and $\Phi(\gamma_{j})$ in terms of the front projection of $\Sigma$ produced by Algorithm 2.  Consider Figure \ref{Fig:UnknotMovie} which shows these curves in the case that $\bar{L}$ is the boundary of a single elementary disk $\disk_{1}$.  The curve $\gamma_{1}$ described in Step 1 of this proof appears on the lower left ($t=0$).  $\Phi(\gamma_{1})$ appears on the lower right ($t=1$).  Note that $\Phi(\gamma_{1})=D^{+}(\gamma_{1})$.

\begin{figure}[h]
    \begin{overpic}[scale=.7]{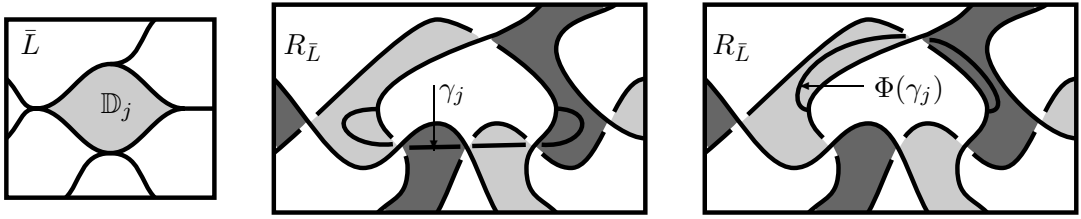}
        \put(1.5,15){$\bar{L}$}
        \put(9,9.5){$\disk_{j}$}
        \put(26,15){$R_{\bar{L}}$}
        \put(66,15){$R_{\bar{L}}$}
        \put(40.5,11){$\gamma_{j}$}
        \put(40,12){\vector(0,-1){6}}
        \put(81,11){$\Phi(\gamma_{j})$}
        \put(80,12){\vector(-1,0){6}}
    \end{overpic}
    \caption{The arcs $\gamma_{j}$ and $\Phi(\gamma_{j})$ associated to an elementary disk $\disk_{j}$.  In the picture we show a close-up of the 1-skeleton $\bar{L}$ of a contact cell decomposition created using Algorithm 2 near an elementary disk.  The ribbon of $\bar{L}$ is labeled $R_{\bar{L}}$.}
    \label{Fig:ArcBasis}
\end{figure}

An arbitrary elementary disk is the same, except that now $\bar{L}$ may have additional Legendrian arcs emanating from $C_{j}=\partial \disk_{j}$.  As these arcs do not intersect $\disk_{j}$ we again conclude that $\Phi(\gamma_{j})=D^{+}_{C_{j}}(\gamma_{j})$.  See Figure \ref{Fig:ArcBasis}.\\

\noindent\textbf{Step 3.}  To finish the proof it suffices to show that if $j>i$ then $D^{+}_{C_{j}}(\gamma_{i})=\gamma_{i}$ and $D^{+}_{C_{i}}\circ D^{+}_{C_{j}}(\gamma_{j}) = D^{+}_{C_{j}}(\gamma_{j})$.  This follows immediately from our choice of indexing.  If $\disk_{j}>\disk_{i}$, then $C_{j}\cap\gamma_{i}=\emptyset$ and $C_{i}\cap\Phi(\gamma_{j}) = \emptyset$.
\end{proof}

\subsection{Controlling the Euler characteristic and number of boundary components}\label{Sec:TakingControl}

In this section we show how the choices involved in carrying out Algorithm 2 can be made so as to establish Equation \ref{Eq:Ineq}.  We also show that the choices may be made so that the surface $\Sigma$ has connected boundary.  Again $L$ will denote a Legendrian link in $\Sthree$ with front projection $D(L)$ to which we apply Algorithm 2.

\begin{lemma}  Assuming that $D(L)$ is non-split, the choices in Algorithm 2 can be made so that the surface $\Sigma$ constructed satisfies
\begin{equation*}
-\chi(\Sigma) = \#(\text{Crossings of $D(L)$}) + \frac{1}{2}\cdot\#(\text{Cusps of $D(L)$}) -1.
\end{equation*}
\end{lemma}

\begin{proof}
The Euler characteristic of the page $\Sigma$ of our open book is given by one minus the total number of elementary disks.  This is a simple consequence of the fact that $\Sigma$ deformation retracts onto a wedge of these cycles.  Again consider the graph $L''$ from Step 2 of Algorithm 2 and let $\lbrace C_{j} \rbrace$ denote its collection of small cycles.  Note that each crossing in $L$ contributes a $-\frac{1}{2}$ to each of exactly two small cycles, and that each cusp of $L$ contributes $-\frac{1}{2}$ to exactly one small cycle.  Therefore
\begin{equation*} \label{Eq:EulerBound}
\sum tb(C_{j}) = -\#(\text{Crossings of $D(L)$}) - \frac{1}{2}\cdot\#(\text{Cusps of $D(L)$}).
\end{equation*}
In completing $L''$ (the graph constructed from L at the end of Step 2) to the 1-skeleton of a contact cell decomposition in Step 3 we added arcs to the $C_{j}$ so as to increase all of their Thurston-Bennequin invariants to $-1$.  For each $j$, it is obvious that this can be done with $-tb(C_{j})$ many non-destabilizeable arcs as $C_{j}$ is embedded in the front projection with no crossings.  We conclude that the entire construction can be carried out so that there is a total number of $-\sum tb(C_{j})$ elementary cycles.
\end{proof}

\begin{rmk} The proof above indicated that for some front diagrams, Step 3 of Algorithm 2 can be carried out in such a way that improves the Euler characteristic bound of Equation \ref{Eq:Ineq}.  See Section \ref{Sec:L41} for an example.
\end{rmk}

\begin{lemma} \label{Thm:ConnectedBinding}
By adding sufficiently many Legendrian arcs in Step 2 of the algorithm in the appropriate manner, we may assume that the binding of the open book constructed is connected.
\end{lemma}

\begin{figure}[h]
    \begin{overpic}[scale=.7]{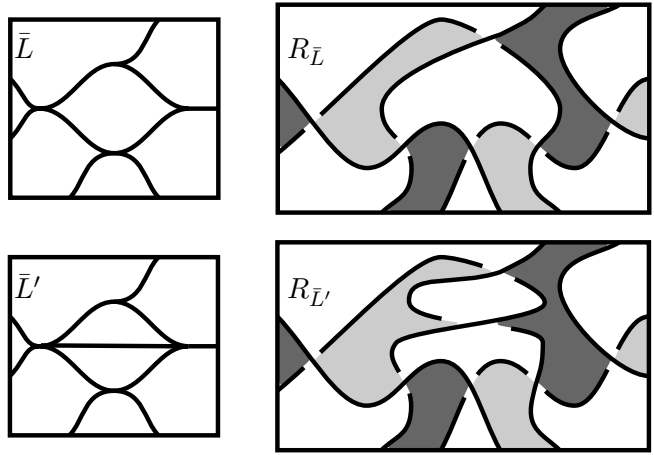}
        \put(1,61.5){$\bar{L}$}
        \put(1,24){$\bar{L}'$}
        \put(43.5,61.5){$R_{\bar{L}}$}
        \put(43.5,24){$R_{\bar{L}'}$}
    \end{overpic}
    \caption{The proof of Theorem \ref{Thm:ConnectedBinding}.}
    \label{Fig:ConnectedBinding}
\end{figure}

\begin{proof}
Let $\Sigma$ be a surface constructed from Algorithm 2 as the ribbon of a Legendrian graph $\bar{L}$.  Consider the arc basis $\lbrace \gamma_{j} \rbrace$ of $\Sigma$ from the proof of Theorem \ref{Thm:Mono}.  As $\Sigma\setminus(\cup\gamma_{j})$ is a disk, $\partial \Sigma$ is connected if and only if for all $j$, the endpoints of the arc $\gamma_{j}$ lie on the same boundary component of $\Sigma$.  Suppose that the number of boundary components of $\Sigma$ is $\geq 2$.  Then for some $j$ the boundary points of the arc $\gamma_{j}$ live on distinct boundary components of $\Sigma$.  Extend $\bar{L}$ to a Legendrian graph $\bar{L}'$ by adding a non-destabilizeable Legendrian arc to $\bar{L}$ connecting the cusps on the boundary of the elementary disk $\disk_{j}$ as in Figure \ref{Fig:ConnectedBinding}.  Then the ribbon $\Sigma'$ of $\bar{L}'$ is again the page of an open book supporting $\Sthree$ with $\#(\partial \Sigma') = \#(\partial \Sigma) -1$.  Therefore the proposition follows by inducting on $\#(\partial \Sigma)$.
\end{proof}

\section{Applications to the study of support invariants and overtwisted surgery diagrams}\label{Sec:SupportInv}
In this section we apply Theorem \ref{Thm:AlgLink} to study support invariants of, as well as the detect overtwisted disks in contact manifolds obtained by contact surgery on Legendrian links in $\Sthree$.  We begin by stating a priori bounds on the support genus and norm of these manifolds in terms of classical link data from Equation \ref{Eq:Ineq}.

\subsection{A priori bounds from surgery}
For a Legendrian knot $K$, let $C(K)$ denote the minimal possible number of crossings of all front projection diagrams of $K$.

\begin{cor} \label{Thm:Bounds}
Let $K$ be a Legendrian knot in $\Sthree$, and suppose that $\Mxi$ is a contact manifold obtained by contact Dehn surgery with coefficient $\frac{1}{k}$ on the $K$ for some $k\in\mathbb{Z}$.  Then
\begin{equation*}
\begin{gathered}
sn(S^{3},\xi_{std},K), \; sn\Mxi \leq 2C(K) - tb(K)+1,\; \text{and}\\
sg(S^{3},\xi_{std},K), \; sg\Mxi \leq C(K) -\frac{1}{2}tb(K) +1.
\end{gathered}
\end{equation*}
\end{cor}

\begin{proof}
Note that by Theorem \ref{Thm:TwistSurgery} it suffices to prove the inequalities for $sn(S^{3},\xi_{std},K)$ and $sg(S^{3},\xi_{std},K)$.  Let $D(K)$ be a front diagram for $K$.  By Theorem \ref{Thm:AlgLink} and Equation \ref{Eq:Tb} there is an open book $\AOB$ supporting $\Sthree$ which contains $K$ in a single page satisfying
\begin{equation*}
\begin{aligned}
-\chi(\Sigma) & = 2\#(\text{Positive crossings of $D(K)$}) -tb(K) +1 \\
& \leq 2\#(\text{Crossings of $D(K)$}) -tb(K) +1.
\end{aligned}
\end{equation*}
Consideration of the above inequality over all possible front projection diagram establishes the bound on the support norm.  As for the support genus, consider the fact that $\Sigma$ has at least one boundary component and apply the above inequality together with
\begin{equation*}
-\chi (\Sigma) = 2g(\Sigma) -2 + \#(\partial\Sigma).
\end{equation*}
\end{proof}

\subsection{Legendrian torus knots}\label{Sec:TorusKnots}

Figure \ref{Fig:TorusKnots} displays front projections of Legendrian $(2n+1,2)$-torus knots ($n\geq 1$) in $\Sthree$.  According to the classification of Legendrian torus knots in \cite{EtHo:KnotsI}, these are the only non-destabilizeable Legendrian representatives (up to Legendrian isotopy).

\begin{figure}
	\begin{overpic}[scale=.7]{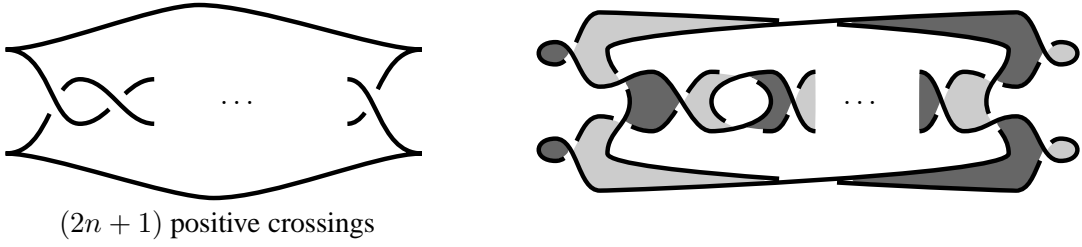}
        \put(78,9){$\dots$}
        \put(20,9){$\dots$}
        \put(5,-3){$(2n+1)$ positive crossings}
    \end{overpic}
    \vspace{4mm}
    \caption{On the left is a Legendrian $(2n+1,2)$-torus knots $K_{n}$.  On the right is a page of an open book supporting $\Sthree$ and containing $K_{n}$, constructed from Algorithm 2.}
    \label{Fig:TorusKnots}
\end{figure}

\begin{thm}  Let $K\subset\Sthree$ be a Legendrian $(2n+1,2)$-torus knot with Thurston-Bennequin number $tb(K)=2n-1-p$ for some $p\in\mathbb{Z}$.  Then $K$ is contained in the page $\Sigma$ of an open book $\AOB$ supporting $\Sthree$ where $\Sigma$ has the topological type of a $(2n+1+p)$-punctured torus.  Therefore
\begin{equation*}
sg(S^{3},\xi_{std},K)\leq 1\quad\text{and}\quad sn(S^{3},\xi_{std},K)\leq 2n+1+p.
\end{equation*}
\end{thm}

\begin{proof}
Applying Algorithm 2 to the front projection of the Legendrian knot $K_{n}$ of Figure \ref{Fig:TorusKnots} will embed $K_{n}$ in the page $\Sigma'$ of an open book $(\Sigma',\Phi')$ supporting $\Sthree$ which has the topological type of a $(2n+1)$-punctured torus.  By the classification of torus knots \cite{EtHo:KnotsI}, $K$ may be obtained from $K_{n}$ be $p$ positive and negative stabilizations.  Then an open book containing $K$ can be obtained from $(\Sigma',\Phi')$ by stabilizing the page $p$ times as in Figure \ref{Fig:Stabilize}.  As the boundary of $\Sigma'$ is disconnected, it can be arranged that each of these stabilizations preserves the genus.
\end{proof}

\subsection{Tight contact structures on $L(4,1)$}\label{Sec:L41}

\begin{figure}[h]
	\begin{overpic}[scale=.7]{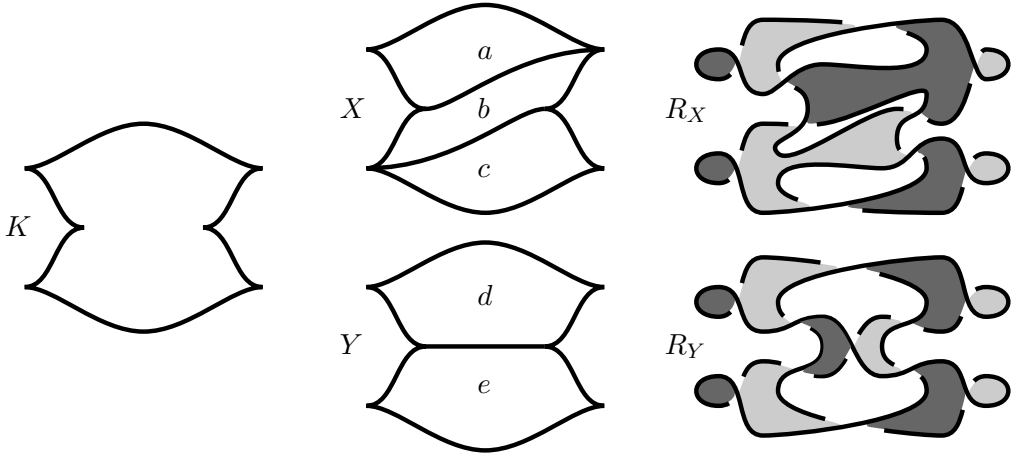}
        \put(-2,22){$K$}
        \put(32,10){$Y$}
        \put(32,34){$X$}
        \put(65,10){$R_{Y}$}
        \put(65,34){$R_{X}$}
        \put(46,40){$a$}
        \put(46,34){$b$}
        \put(46,28){$c$}
        \put(46,15){$d$}
        \put(46,6){$e$}
    \end{overpic}
	\caption{$K$ is a Legendrian unknot with $tb(K)=-3$ and $rot(K)=0$.  The middle column shows two contact cell decomposition of $\Sthree$ containing $K$.  The upper one has 1-skeleton $X$ and 2-cells labeled $a$, $b$, and $c$.  The lower one has 1-skeleton $Y$ and 2-cells $d$ and $e$.  The right column shows the ribbons $R_{X}$ and $R_{Y}$ of $X$ and $Y$, respectively.  The ribbon $R_{X}$ has the topological type of a 3-punctured disk and the associated open book $(R_{X},\Phi_{X})$ of $\Sthree$ has monodromy $\Phi_{X}=D^{+}_{\partial c}\circ D^{+}_{\partial b}\circ D^{+}_{\partial a}$.  The ribbon $R_{Y}$ has the topological type of a punctured torus and the associated open book $(R_{Y},\Phi_{Y})$ of $\Sthree$ has monodromy $\Phi_{Y}=D^{+}_{\partial e}\circ D^{+}_{\partial d}$.}
    \label{Fig:L41Ribbon}
\end{figure}

Figure \ref{Fig:L41Ribbon} displays two distinct contact cell decompositions of $\Sthree$ which contain the Legendrian unknot with $rot=0$ and $tb=-3$.

\begin{thm} \label{Thm:L41}
Let $\xi_{0}$ and $\xi_{1}$ denote the two contactomorphism classes of tight contact structures on the lens space $L(4,1)$ with zero and non-zero Euler classes, respectively (see \cite{Honda:TightClassI}).  Then
\begin{equation*}
\begin{gathered}
sg(L(4,1),\xi_{0})=sg(L(4,1),\xi_{1}) = 0,\quad bn(L(4,1),\xi_{0}) = bn(L(4,1),\xi_{1}) = 4,\\
sn(L(4,1),\xi_{0})= 1,\quad\text{and}\quad sn(L(4,1),\xi_{1})=2.
\end{gathered}
\end{equation*}
\end{thm}

\begin{proof}
We assume familiarity with the classification of tight contact structures on lens spaces.  See \cite{Honda:TightClassI}.  The two diffeomorphism classes of tight contact structures on $L(4,1)$ can be obtained by Legendrian surgery on the unknots whose classical invariants are $tb=-3,\;rot=0$ and $tb=-3,\;rot=\pm2$ in $\Sthree$.  As in the statement of the theorem, we will call these structures $\xi_{0}$ and $\xi_{1}$ respectively.  Note that this differs from the \emph{isotopy} classification, in which case there is no ambiguity in the sign of rotation number.  By starting with the Legendrian unknot with Thurston-Bennequin invariant $-1$ embedded in the page of an annulus open book for $\Sthree$ and applying positive and negative stabilizations as necessary, it follows that every Legendrian unknot can be realized in the page of a planar open book decomposition of $\Sthree$.  Since Legendrian surgery can then be realized as the precomposition of the monodromy of such an open book with a positive Dehn twist about the curve, we see that every tight contact structure on every lens space $L(p,1)$ has support genus zero.

The remainder of the proof is a recollection of results appearing in \cite{EtOzb:Support} together with the existence of the open book described in Figure \ref{Fig:L41Ribbon}.  There it is shown that $bn(L(4,1),\xi_{0})=4$.  This is a consequence of the fact that a planar  open book for $(L(4,1),\xi_{0})$ with four binding components exists, along with the following two observations: (1) Any contact structure on $L(p,1)$ with $p>2$ which is supported on an annulus must be overtwisted. (2)  If a tight contact manifold is given as an open book decomposition whose page is a twice punctured disk and whose first homology has order four, then it must be either $L(4,3)$ or $L(2,1)\# L(2,1)$.  Therefore, the tightness of $\xi_{1}$ and existence of a planar open book with four binding components (a disks with three punctures whose monodromy is a positive Dehn twist about every boundary component as in the previous paragraph) implies that same the proof can be applied verbatim to show that $bn(L(4,1),\xi_{1})=4$.

It only remains to calculate the support norms.  Any contact manifold $\Mxi$ for which $c(\xi)\neq 0$ cannot be supported by an open book whose page is a punctured torus.  See Remark \ref{Rmk:ChernClass}.  This, together with the existence of an open book decomposition whose page has Euler characteristic $-2$ and the above remark regarding annular open books of lens spaces leads to the conclusion that $sn(L(4,1),\xi_{1}) = 2$.  Figure \ref{Fig:L41Ribbon} shows that $sn(L(4,1),\xi_{0})=1$
\end{proof}

\subsection{Stabilization in surgery diagrams} \label{Sec:PosStab}

In this section we equate overtwistedness of a contact manifold $\Mxi$ with the existence of a certain type of surgery diagram determining $\Mxi$. This correspondence, stated in Theorem \ref{Thm:Destab}, may be viewed as a surgery theoretic interpretation of Theorem \ref{Thm:GirouxStab}(2).  Throughout, $L=L^{+}\cup L^{-}$ be a contact surgery diagram for the contact 3-manifold $\Mxi$.

\begin{defn}
For a Legendrian knot $K\subset L$, a \emph{standard meridian} of $K$ is a Legendrian unknot $\mu_{K}$ with Thurston-Bennequin invariant $-1$ which is topologically a meridian for $K$ and is not knotted with any other component of $L$.
\end{defn}
\begin{figure}[ht]
\begin{overpic}[scale=.75]{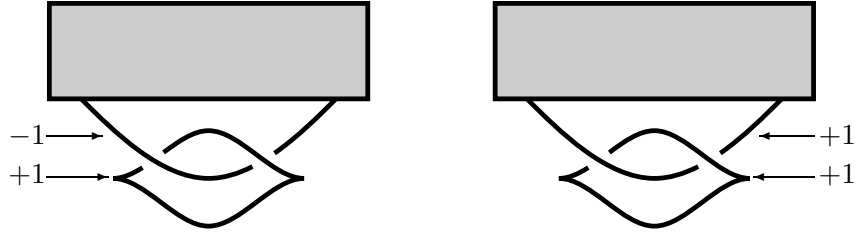}
    \put(-5,11){$-1$}
        \put(0,12){\vector(1,0){7.25}}
    \put(-5,6){$+1$}
        \put(0,6.75){\vector(1,0){8}}
    \put(100.5,11){$+1$}
        \put(100,12){\vector(-1,0){7.25}}
    \put(100.5,6){$+1$}
        \put(100,6.75){\vector(-1,0){8}}
\end{overpic}
\caption{The gray boxes represent an ambiguity as to what the remainder of the diagrams look like.  On the left is a positive stabilization as in Theorem \ref{Thm:Destab}(1), and on the right is a negative stabilization as in Theorem \ref{Thm:Destab}(2).  The contact manifold described by the surgery diagram on the right figure is overtwisted.}
\label{Fig:Stabs}
\end{figure}

\begin{thm} \label{Thm:Destab}
Let $\Mxi$ be a contact 3-manifold.
\be
\item If a surgery diagram $L$ for $\Mxi$ is such that there exists a knot, $K\subset L^{-}$, which has a standard meridian, $\mu_{K}\subset L^{+}$, then $K\cup \mu$ may be deleted from the diagram.
\item $\Mxi$ is overtwisted if and only if it admits a contact surgery diagram $L$, such that there is a knot, $K$ in $L^{+}$, and a standard meridian for $K$ which is also in $L^{+}$.
\ee
\end{thm}

\begin{rmk} Theorem \ref{Thm:Destab} (a) can also be recovered from Proposition 2 in \cite{DG:Handles} (a handle slide) together with the authors' well known cancelation lemma.  See also Section \ref{Sec:KirbyEx}.  Their proof, which provides a more general result, relies on a convex surface argument.
\end{rmk}

First we consider the existence result in part (b).  This is obtained by a simple application of Theorem \ref{Thm:GirouxStab}.

\begin{lemma}
Suppose that $\Mxi$ is overtwisted.  Then it admits a surgery diagram $L\subset\Sthree$ for which the underlying link is split.  One of these split components is an unknot whose Thurston-Bennequin number is $-1$ together with its standard meridian, both of which have surgery coefficient $+1$.
\end{lemma}

\begin{proof}
By Theorem \ref{Thm:GirouxStab} there is an open book $\AOB$ whose negative stabilization is equivalent to $\Mxi$.  Let $(M,\xi ')$ be the contact structure on $M$ determined by $\AOB$, and let $\SOT = \OBOT$ be the overtwisted contact structure on the sphere determined by the open book whose page is an annulus and whose monodromy is a negative Dehn twist about a curve parallel to one of its boundary components.  From the open book perspective, $\Mxi$ is a Murasugi sum of $\AOB$ and $\OBOT$, and hence from the 3-manifold perspective is the contact connected sum of $(M,\xi ')$ with $\SOT$.  Apply Corollary \ref{Thm:ContLickWall} to $(M,\xi ')$ to obtain an equivalent surgery diagram $L'$.  $\SOT$ may be presented by Legendrian $+1$ surgery on the split component as described in the statement of the theorem.  Then contact surgery on the disjoint union of $L'$ and the split component will yield $(M,\xi ')\# \SOT =\Mxi$.
\end{proof}

The proof of the preceding lemma along with Theorem $\ref{Thm:GirouxStab}$ indicate that the existence of the type of surgery diagram described is also equivalent to overtwistedness.  This type of diagram may be likened to negatively stabilizing an open book by boundary connect summing its page with $\OBOT$.  Much wilder negative stabilizations are of course possible, by the fact that there is freedom in the choice of arc that can be used to perform a Murasugi sum.

\begin{lemma}
Suppose that $\Mxi$ is given by contact surgery on $L=L^{+}\cup L^{-} \subset \Sthree$.  Assume that there is a knot $K\subset L^{+}$ for which a standard meridian of $K$ is also in $L^{+}$.  Then there is a compatible open book decomposition for $\Mxi$ which is a negative stabilization of some other open book decomposition of $M$.  Hence, $\Mxi$ is overtwisted.
\end{lemma}

\begin{proof}
Let $\mu_{K}$ be a standard meridian of $K$.  After possibly applying a series of Legendrian Reidemeister moves and isotopies, we may assume that the height function $z$ on $L\setminus\mu_{K}$ (considered as living in $\Rthree$) takes on its absolute minimum at some point on $K$.  We may also work under the assumption that $\mu_{K}$ is contained in a small ball near this point as depicted in Figure \ref{Fig:Stabs}.

Without loss of generality, the point on $K$ for which $z$ achieves its absolute minimum has $z$-value equal a zero.  Draw a straight, horizontal line connecting the two points on $K$ which take on the $z$-value $\epsilon > 0$.  Note that for $\epsilon$ sufficiently small, these two points are uniquely determined, and that such an arc will not cross any other components of $L\setminus\mu$ in the diagram.  Isotop this arc, while fixing its endpoints so that its union with $(z|_{K})^{-1}[0,\epsilon]$ is a Legendrian unknot $\gamma$ with $tb=-1$, and still does not cross any other component of $L\setminus\mu$.  Call this arc $\gamma$.

Follow the proof of Theorem \ref{Thm:AlgLink} to complete $(L\setminus\mu_{K})\cup\gamma$ to a contact cell decomposition of $\Sthree$.  In taking the ribbon of the 1-skeleton of this cell decomposition, we have the page, $\Sigma$ of an open book for $\Sthree$ which contains $L\setminus\mu_{K}$.  We are interested especially in the 1-handle of this surface which is the ribbon of $(z|_{K})^{-1}[0,\epsilon]$.  Denote this 1-handle by $H$.  By Theorem \ref{Thm:Mono} the monodromy of this open book may be written
\begin{equation*}
\Phi = D^{+}_{\gamma} \circ \prod D^{+}_{\zeta_{j}}
\end{equation*}
where each of the $\zeta_{j}$ are curves which do not intersect the co-core of $H$.  Note also by our construction that $K$ is the only component of $L\setminus\mu_{K}$ which intersects the co-core of $H$.

Write $\Sigma_{t}$ and $\zeta_{t}$ for the image of $\Sigma$ and an arc $\zeta$ contained in it under the time-$t$ flow of $\partial_{z}$.  Note that for $\delta >0$ sufficiently small we may realize $\mu$ as $\gamma_{-\delta}$.  Then the monodromy, $\Phi_{L}$ of the open book obtained by performing contact Dehn surgery on $L$ may be written
\begin{equation*}
D^{+}_{\gamma} \circ \prod D^{+}_{\zeta_{j}} \circ \prod_{L\setminus(\mu\cup K)} D^{\delta_{j}}_{\eta_{j}} \circ D^{-}_{K} \circ D^{-}_{\gamma}
\end{equation*}
as $\mu_{K}$ projects to $\gamma$ in $\Sigma$.  The open book is then the same as is given by the mapping class
\begin{equation*}
\prod D^{+}_{\zeta_{j}} \circ \prod_{L\setminus(\mu\cup K)} D^{\delta_{j}}_{\eta_{j}} \circ D^{-}_{K}
\end{equation*}
as the two are conjugate.  Then by the properties established concerning the intersections of the $\zeta_{j}$ and $\eta_{j}$ with $H$, we have that this open book decomposition is a negative stabilization of the open book whose page is $\Sigma\setminus H$ and whose monodromy is
\begin{equation*}
\prod D^{+}_{\zeta_{j}} \circ \prod_{L\setminus(\mu\cup K)} D^{\delta_{j}}_{\eta_{j}}.
\end{equation*}
This concludes the proof.
\end{proof}

\begin{proof}[Proof of Theorem \ref{Thm:Destab}]
It only remains to prove part (a).  This can be seen by following the proof of the previous lemma word-for-word with the exception that now $K\subset L^{-}$.  Construct an open book decomposition for $L$ as instructed.  Eliminating the twists about $K$ and $\mu$ provides a new open book decomposition whose monodromy is
\begin{equation*}
D^{+}_{\gamma} \circ \prod D^{+}_{\zeta_{j}} \circ \prod_{L\setminus(\mu\cup K)} D^{\delta_{j}}_{\eta_{j}}
\end{equation*}
which is conjugate to
\begin{equation*}
\prod D^{+}_{\zeta_{j}} \circ \prod_{L\setminus{\mu\cup K}} D^{\delta_{j}}_{\eta_{j}} \circ D^{+}_{\gamma}.
\end{equation*}
Then deleting $D^{+}_{\gamma}$ from the monodromy and $H$ from the page amounts to a positive destabilization which does not alter the contact manifold.  By construction, this is equivalent to removing $K$ and $\mu$ from the surgery diagram as we may build the same open book for $L\setminus (K\cup\mu)$ by taking the page to be the same minus the handle $H$.
\end{proof}



\section{Mapping class relations as Kirby moves}\label{Sec:Kirby}

In this section we introduce a method of modifying contact surgery diagrams by constructing Kirby moves  associated to mapping class relations.  We call these moves \emph{ribbon moves}, which were briefly described in the introduction.  After defining these operations, we give examples of ribbon moves analogous to the conjugacy, braid, and chain relations between Dehn twists on a surface.  An example of a lantern relation type ribbon move was given in Figure \ref{Fig:Lantern}.  In Section \ref{Sec:KirbyProof} we show that any two contact surgery diagrams for the same contact manifold are related by a sequence of Legendrian isotopies and ribbon moves.  This may be thought of as Theorem \ref{Thm:GirCor}(1) interpreted in the language of contact surgery by Theorem \ref{Thm:TwistSurgery}.

The following conventions will be used throughout the remainder of the paper:
\be
\item  $R_{G}$ will denote the ribbon of a Legendrian graph $G\subset\Sthree$.
\item  We assume that $R_{G}\times[-1,1]$ is embedded in $\Sthree$ in such a way that a coordinate on $[-1,1]$ coincides with the function $z$ on $\Rthree$ up to multiplication by a positive number.
\item  If $\gamma\subset R_{G}$ is a Legendrian realizable, simple, closed curve, then we write $\gamma(t)$ for the Legendrian realization of $\gamma\times\{t\}\subset R_{G}\times\{t\}$ for $t\in[-1,1]$.
\item  When $L\subset\Sthree$ is a Legendrian link and we write $L\subset R_{G}\times[-1,1]$, it will be assumed that each connected component $L_{j}$ of $L$ is contained in $R_{G}\times\{t_{j}\}$ for some $t_{j}\in[-1,1]$.  We also assume that the $L_{j}$ are indexed so that $j>i$ implies $t_{j}>t_{i}$.
\item  $MCG(R_{G},\partial R_{G})$ will refer to the \emph{mapping class group} of $R_{G}$, i.e. the group of orientation preserving diffeomorphisms of $R_{G}$ which restrict to the identity on a neighborhood of $\partial R_{G}$, considered up to isotopy.
\ee

\subsection{Ribbon equivalence and ribbon moves}

In this section we give a more careful definition of the ribbon moves which were described in the introduction.  We begin with some definitions which will help to translate surgery data into mapping class data and vice versa.

\begin{defn}
Suppose that $R_{G}$ is the ribbon of a Legendrian graph $G\subset\Sthree$ and that
\begin{equation*}
L=\bigcup^{n}_{1} L_{j}^{\delta_{j}}\subset R_{G}\times[-1,1],\quad \delta_{j}\in\{+1,-1\}
\end{equation*}
is a surgery link.  Then each $L_{j}$ may be considered as a simple closed curve in $R_{G}$ by projecting $R_{G}\times[-1,1]$ to $R_{G}$.  The \emph{mapping class determined by $L$} is
\begin{equation*}
D_{L}:=D_{L_{n}}^{-\delta_{n}}\circ\cdots\circ D_{L_{1}}^{-\delta_{1}}\in MCG(R_{G},\partial R_{G}).
\end{equation*}
\end{defn}

\begin{defn}
Suppose that $R_{G}$ is the ribbon of a Legendrian graph in $\Sthree$ and that $L,L'\subset R_{G}\times[-1,1]$ are two surgery links.  We say that $L$ and $L'$ are \emph{$R_{G}$ equivalent} if $D_{L}=D_{L'}$ in $MCG(R_{G},\partial R_{G})$.  Surgery links $L$ and $L'$ in $\Sthree$ are \emph{ribbon equivalent} if there exists a Legendrian graph $G\subset\Sthree$ for which $L$ and $L'$ are $R_{G}$ equivalent.
\end{defn}

\begin{prop}\label{Prop:RibbonEquiv}
If $L$ and $L'$ are ribbon equivalent, then they determine the same contact 3-manifold.
\end{prop}

\begin{proof}
Suppose that $L$ and $L'$ are $R_{G}$ equivalent for some Legendrian graph $G\subset\Sthree$.  By adjoining Legendrian arcs to $G$, we can find a Legendrian graph $\bar{G}$ containing $G$ whose ribbon $R_{\bar{G}}$ is the page of an open book $(R_{\bar{G}},\Phi_{\bar{G}})$ supporting $\Sthree$.  Then $R_{G}$ embeds into $R_{\bar{G}}$ and the mapping classes associated to $L$ and $L'$ extend to mapping classes of $R_{\bar{G}}$ in the obvious way.

The contact manifolds determined by $L$ and $L'$ are then supported by the open books
\begin{equation*}
(R_{\bar{G}},\Phi_{\bar{G}}\circ D_{L})\quad\text{and}\quad (R_{\bar{G}},\Phi_{\bar{G}}\circ D_{L'}),
\end{equation*}
respectively.  Therefore $L$ and $L'$ determine the same contact manifold.
\end{proof}

\begin{defn}
Let $L$ be a contact surgery link in $\Sthree$ with surgery sublink $\ell$.  A \emph{ribbon move} performed on $\ell$ consists of finding a Legendrian graph $G$ such that $\ell\subset R_{G}\times[-1,1]$ and replacing $\ell$ with another surgery link $\ell'\subset R_{G}\times[-1,1]$ for which $\ell$ and $\ell'$ are $R_{G}$ equivalent.  When $G$ is specified, we will call such a move an \emph{$R_{G}$ move}.
\end{defn}

Note that by Proposition \ref{Prop:RibbonEquiv}, performing a ribbon move does not change that contact manifold determined by surgery.

The simplest types of ribbon moves are \emph{insertions and deletions of canceling pairs}, first observed in \cite{DG:Surgery}.  Insertion of a canceling pair consists of adding a Legendrian knot $K$ with surgery coefficient $\pm 1$ to a surgery diagram, together with a Reeb pushoff of $K$ decorated with surgery coefficient $\mp 1$.  A canceling pair determines the mapping class $D^{\mp}_{K}\circ D^{\pm}_{K}=Id_{R_{K}}$ of the ribbon $R_{K}$.  Therefore a canceling pair is ribbon equivalent to the empty surgery link.

Ribbon moves performed along a ribbon $R_{G}$ may be alternatively characterized as the insertion of a surgery link $L\subset R_{G}\times[-1,1]$ into a surgery diagram satisfying $D_{L}=Id_{R_{G}}$, followed by deletions of canceling pairs.

\subsection{Examples of ribbon moves}\label{Sec:KirbyEx}

In this section we give examples of ribbon moves.  We begin by describing a method which allows us to consider distinct Legendrian knots in $\Sthree$ as sharing a transverse intersection on the ribbon of some Legendrian graph.  This will be useful in providing examples of ribbon moves, as mapping class relations often relate Dehn twists along curves which intersect nontrivially.

\subsubsection{Plumbing ribbons of Legendrian knots}

Consider two Legendrian knots $A$ and $B$ in $\Sthree$ whose front projections are in one of the two configurations shown in Figure \ref{Fig:KnotConfig}.

\begin{figure}[h]
	\begin{overpic}[scale=.7]{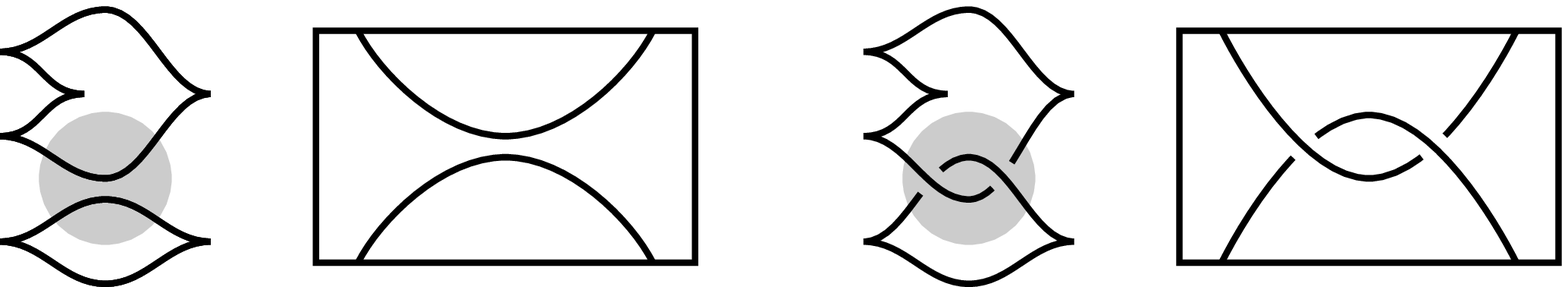}
        \put(-3,9){$A$}
        \put(-3,2){$B$}
        \put(11.5,6){\vector(1,0){7}}
        \put(10,-3){Configuration 1}
        \put(52.5,9){$A$}
        \put(52.5,2){$B$}
        \put(67.5,6){\vector(1,0){7}}
        \put(66,-3){Configuration 2}
    \end{overpic}
    \vspace{3mm}
	\caption{The boxes show the front projection of the knots $A$ and $B$ ``zoomed in'' at the gray disks.}
    \label{Fig:KnotConfig}
\end{figure}

By pinching $A$ and $B$ together along a chord of the vector field $\partial_{z}$ we obtain a Legendrian graph $A\vee B$ which is homeomorphic to a wedge of two circles.  Then the ribbon $R_{A\vee B}$ of $A\vee B$ is a plumbing of $R_{A}$ of $R_{B}$ and is diffeomorphic to a punctured torus.  See Figure \ref{Fig:Pinching}.  We have $A\subset R_{A\vee B}\times[-1,1]$ and $B\subset R_{A\vee B}\times[-1,1]$.

\begin{figure}[h]
	\begin{overpic}[scale=.7]{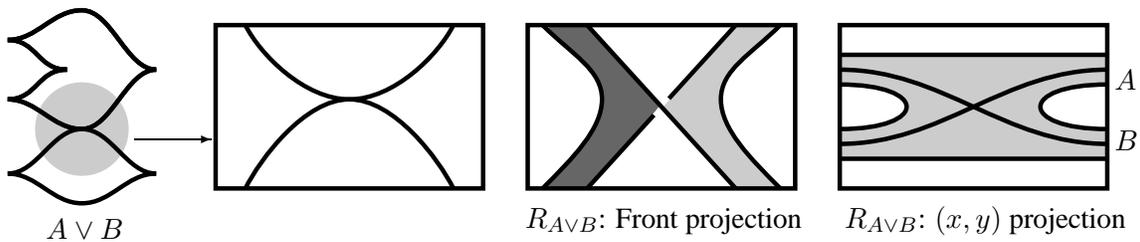}
        \put(3.5,-3){$A\vee B$}
        \put(11.5,6){\vector(1,0){7}}
        \put(47,-2){$R_{A\vee B}$: Front projection}
        \put(76,-2){$R_{A\vee B}$: $(x,y)$ projection}
        \put(100.5,10.5){$A$}
        \put(100.5,5){$B$}
    \end{overpic}
    \vspace{3mm}
	\caption{The graph $A\vee B$ and its ribbon in the front and $(x,y)$ projections.  The right-most figure shows the curves $A$ and $B$ projected to $R_{A\vee B}$.}
    \label{Fig:Pinching}
\end{figure}

Suppose that $A$ and $B$ are decorated with surgery coefficients $\delta_{A}$ and $\delta_{B}$.  Then in the first configuration of Figure \ref{Fig:KnotConfig}, $D_{A^{\delta_{A}}\cup B^{\delta_{B}}} = D_{A}^{-\delta_{A}}\circ D_{B}^{-\delta_{B}}$.  In the second configuration of Figure \ref{Fig:KnotConfig}, $D_{A^{\delta_{A}}\cup B^{\delta_{B}}} = D_{B}^{-\delta_{B}}\circ D_{A}^{-\delta_{A}}$.

\begin{figure}[h]
	\begin{overpic}[scale=.7]{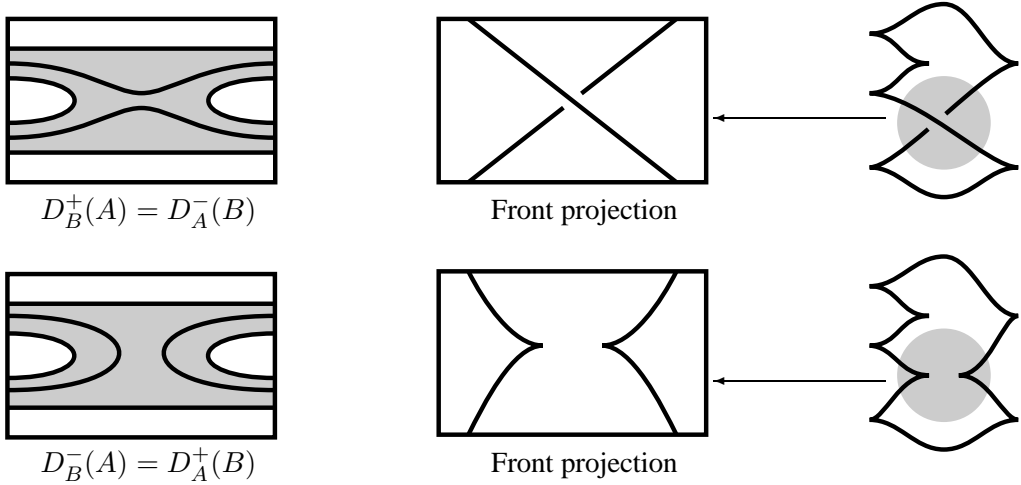}
        \put(3.5,23){$D^{+}_{B}(A)=D^{-}_{A}(B)$}
        \put(48,23){Front projection}
        \put(3.5,-2){$D^{-}_{B}(A)=D^{+}_{A}(B)$}
        \put(48,-2){Front projection}
        \put(87,33){\vector(-1,0){17}}
        \put(87,7){\vector(-1,0){17}}
    \end{overpic}
    \vspace{1.5mm}
	\caption{The left-most column shows the curves $D^{+}_{B}(A)=D^{-}_{A}(B)$ and $D^{-}_{B}(A)=D^{+}_{A}(B)$ embedded into $R_{A\vee B}$ in the $(x,y)$ projection.  The remaining columns show the Legendrian realizations of these curves in the front projection.}
    \label{Fig:Realize}
\end{figure}

The curves $D^{+}_{B}(A)=D^{-}_{A}(B)$ and $D^{-}_{B}(A)=D^{+}_{A}(B)$ can also be Legendrian realized in $R_{A\vee B}\times[-1,1]$ as in Figure \ref{Fig:Realize}.

\subsubsection{Handle slides as ribbon moves}\label{Sec:HandleSlide}

Using the local pictures above, we can interpret the handle slides of \cite{DG:Handles} as ribbon moves.  Suppose that $\alpha$ and $\beta$ are two simple closed curves on an oriented surface $\Sigma$ which share a single transverse intersection.  The \emph{conjugacy relation} tells us that
\begin{equation*}
D^{+}_{\alpha}\circ D^{\pm}_{\beta} = D^{\pm}_{D^{+}_{\alpha}(\beta)}\circ D^{+}_{\alpha}.
\end{equation*}

Using the curves shown in Figure \ref{Fig:Realize} we can associate a ribbon move to the conjugacy relation and its variants.  Examples are given in Figure \ref{Fig:HandleSlide}.

\begin{figure}[h]
	\begin{overpic}[scale=.7]{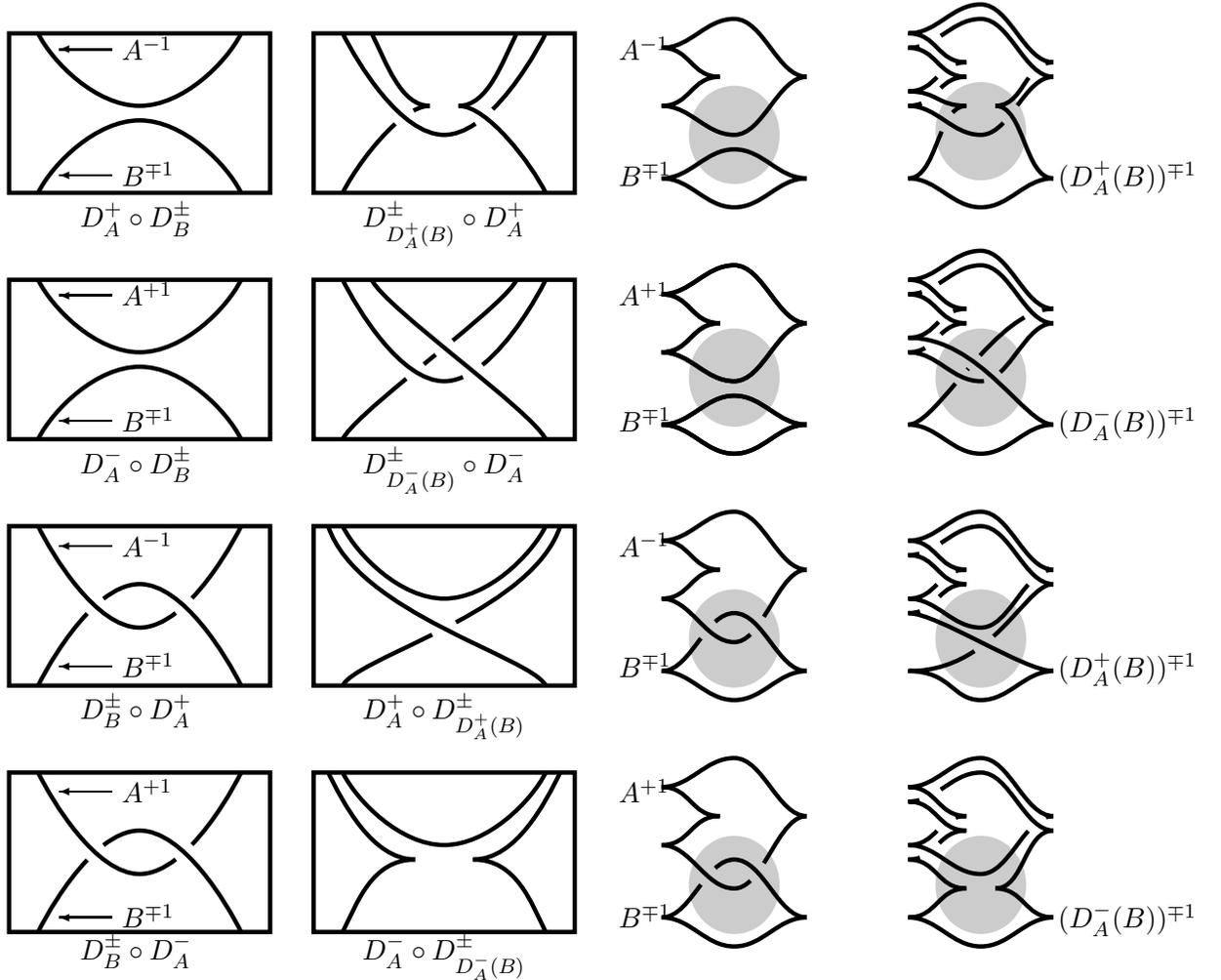}
            \put(7,69){$D^{+}_{A}\circ D^{\pm}_{B}$}
        \put(34,69){$D^{\pm}_{D^{+}_{A}(B)}\circ D^{+}_{A}$}
        \put(11,85){$A^{-1}$}
        \put(58.5,85){$A^{-1}$}
        \put(10,86){\vector(-1,0){5}}
        \put(11,73){$B^{\mp 1}$}
        \put(58.5,73){$B^{\mp 1}$}
        \put(100.5,73){$(D^{+}_{A}(B))^{\mp 1}$}
        \put(10,74){\vector(-1,0){5}}
            \put(7,45.5){$D^{-}_{A}\circ D^{\pm}_{B}$}
        \put(34,45.5){$D^{\pm}_{D^{-}_{A}(B)}\circ D^{-}_{A}$}
        \put(11,61.5){$A^{+1}$}
        \put(58.5,61.5){$A^{+1}$}
        \put(10,50.5){\vector(-1,0){5}}
        \put(11,49.5){$B^{\mp 1}$}
        \put(58.5,49.5){$B^{\mp 1}$}
        \put(100.5,49.5){$(D^{-}_{A}(B))^{\mp 1}$}
        \put(10,62.5){\vector(-1,0){5}}
            \put(7,22){$D^{\pm}_{B}\circ D^{+}_{A}$}
        \put(34,22){$D^{+}_{A}\circ D^{\pm}_{D^{+}_{A}(B)}$}
        \put(11,37.5){$A^{-1}$}
        \put(58.5,37.5){$A^{-1}$}
        \put(10,38.5){\vector(-1,0){5}}
        \put(11,26){$B^{\mp 1}$}
        \put(58.5,26){$B^{\mp 1}$}
        \put(100.5,26){$(D^{+}_{A}(B))^{\mp 1}$}
        \put(10,27){\vector(-1,0){5}}
            \put(7,-1){$D^{\pm}_{B}\circ D^{-}_{A}$}
        \put(34,-1){$D^{-}_{A}\circ D^{\pm}_{D^{-}_{A}(B)}$}
        \put(11,14){$A^{+1}$}
        \put(58.5,14){$A^{+1}$}
        \put(10,15){\vector(-1,0){5}}
        \put(11,2){$B^{\mp 1}$}
        \put(58.5,2){$B^{\mp 1}$}
        \put(100.5,2){$(D^{-}_{A}(B))^{\mp 1}$}
        \put(10,3){\vector(-1,0){5}}
    \end{overpic}
	\caption{Each row gives an example of handle slide type ribbon move.  In each example the surgery curve $B$ is slid over the curve $A$.  The two boxes in each row show ribbon equivalent surgery curves inside of a Darboux ball as in Figures \ref{Fig:KnotConfig} and \ref{Fig:Realize}.  Below each box is the mapping class of $R_{A\vee B}$ associated to the surgery link.  The two surgery diagrams in each row show how the surgery curve $B$ is modified by the ribbon move.}
    \label{Fig:HandleSlide}
\end{figure}

\subsubsection{A braid relation for surgery links}

Again suppose that $\alpha$ and $\beta$ are simple, closed curves on an oriented surface $\Sigma$ sharing a single transverse intersection.  The \emph{braid relation} is the equation
\begin{equation*}
D^{\pm}_{\alpha}\circ D^{\pm}_{\beta}\circ D^{\pm}_{\alpha} = D^{\pm}_{\beta}\circ D^{\pm}_{\alpha} \circ D^{\pm}_{\beta}.
\end{equation*}

For two Legendrian knots $A$ and $B$ which share a chord as in Figure \ref{Fig:KnotConfig}, we can use the surface $R_{A\vee B}$ to define a ribbon move associated to the braid relation.  The braid relation, stated in terms of contact surgery, says that the surgery links
\begin{equation*}
A^{\mp 1}(-1)\cup B^{\mp 1}(0) \cup A^{\mp 1}(+1)\subset R_{A\vee B}\times [-1,1]\quad\text{and}\quad B^{\mp 1}(-1)\cup A^{\mp 1}(0) \cup B^{\mp 1}(+1)\subset R_{A\vee B}\times [-1,1]
\end{equation*}
are $R_{A\vee B}$ equivalent.  An example is given in Figure \ref{Fig:Braid}.

\begin{figure}[h]
	\begin{overpic}[scale=.7]{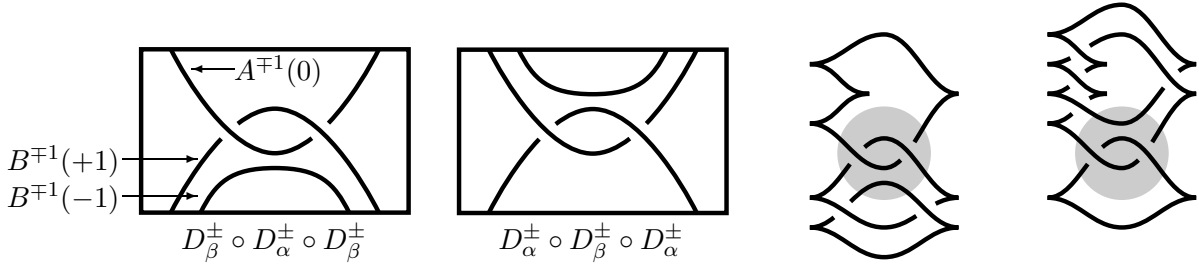}
        \put(4,1){$D^{\pm}_{\beta}\circ D^{\pm}_{\alpha} \circ D^{\pm}_{\beta}$}
        \put(34,1){$D^{\pm}_{\alpha}\circ D^{\pm}_{\beta}\circ D^{\pm}_{\alpha}$}
        \put(-12.5,5){$B^{\mp 1}(-1)$}
        \put(-1.5,6){\vector(1,0){7}}
        \put(-12.5,8.5){$B^{\mp 1}(+1)$}
        \put(-1.5,9.5){\vector(1,0){7}}
        \put(9,17){$A^{\mp 1}(0)$}
        \put(9,18){\vector(-1,0){4}}
    \end{overpic}
    \vspace{1.5mm}
	\caption{The left-most box shows the surgery link $B^{\mp 1}(-1)\cup A^{\mp 1}(0) \cup B^{\mp 1}(+1)\subset R_{A\vee B}\times [-1,1]$.  The center-left box shows the surgery link $A^{\mp 1}(-1)\cup B^{\mp 1}(0) \cup A^{\mp 1}(+1)\subset R_{A\vee B}\times [-1,1]$.  Again, the mapping class of $R_{A\vee B}$ associated to the surgery links are shown below each box.  An explicit example of the corresponding surgery diagrams are shown on the right.  For the two surgery diagrams, all surgery coefficients are the same and are either $-1$ or $+1$.}
    \label{Fig:Braid}
\end{figure}

\subsubsection{A 3-chain relation}

For our final example, we present a ribbon move associated to the 3-chain relation.

\begin{figure}[h]
	\begin{overpic}[scale=.7]{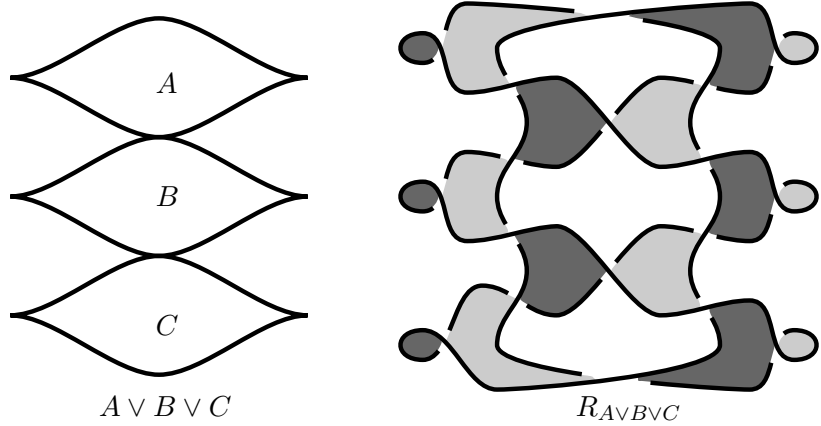}
        \put(18,37){$A$}
        \put(18,23){$B$}
        \put(18,7){$C$}
        \put(11,-3){$A\vee B\vee C$}
        \put(70,-3){$R_{A\vee B\vee C}$}
    \end{overpic}
    \vspace{3mm}
	\caption{On the left is a Legendrian graph $A\vee B\vee C$, whose cycles consist of Legendrian unknots $A,B$ and $C$, each of which has Thurston-Bennequin number $-1$.  On the right is the ribbon $R_{A\vee B\vee C}$ of the graph.  Note that $R_{A\vee B\vee C}$ has the topological type of a twice punctured torus.}
    \label{Fig:2Punctured}
\end{figure}

Consider the Legendrian graph $A\vee B\vee C$ shown in Figure \ref{Fig:2Punctured}.  The ribbon $R_{A\vee B\vee C}$ of the graph has the topological type of a twice punctured torus.  Then $A,B$ and $C$ may be considered as representing simple, closed curves in $R_{A\vee B\vee C}$.  Let $X$ and $Y$ denote curves in $R_{A\vee B\vee C}$ which are isotopic to the boundary components of $R_{A\vee B\vee C}$.  The \emph{3-chain relation} is the mapping class relation
\begin{equation*}
(D^{+}_{C}\circ D^{+}_{B}\circ D^{+}_{A})^{4} = D^{+}_{X}\circ D^{+}_{Y}.
\end{equation*}

\begin{figure}[h]
	\begin{overpic}[scale=.7]{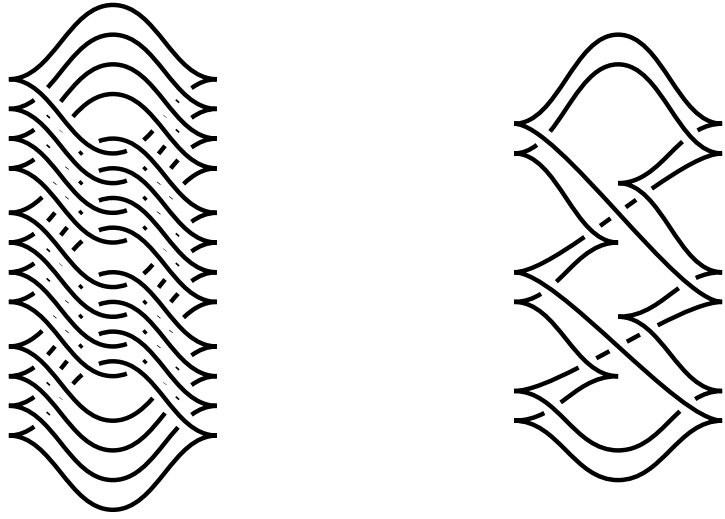}
    \end{overpic}
    \vspace{1.5mm}
	\caption{On the left is the surgery link L.  On the right is the surgery link $L'$.  In this picture all surgery coefficients are $-1$.}
    \label{Fig:Chain}
\end{figure}

We can associate surgery links to each sides of the above equation using the curves $A,B,C$ and Legendrian representatives of $X$ and $Y$.  Consider the surgery link
\begin{equation*}
L=\bigcup_{j=0}^{3}(C^{-1}(1-\frac{j}{2})\cup B^{-1}(1-\frac{j}{2}-\frac{1}{6}) \cup A^{-1}(1-\frac{j}{2}-\frac{1}{3})).
\end{equation*}
Then the mapping class of $R_{A\vee B\vee C}$ associated to $L$ is
\begin{equation*}
D_{L}=(D^{+}_{C}\circ D^{+}_{B}\circ D^{+}_{A})^{4}.
\end{equation*}
Consider also the surgery link
\begin{equation*}
L'=X^{-1}(1)\cup Y^{-1}(-1)\quad\text{with}\quad D_{L'}=D^{+}_{X}\circ D^{+}_{Y}.
\end{equation*}
Then by the 3-chain relation, $L\subset R_{A\vee B\vee C}$ and $L'\subset R_{A\vee B\vee C}$ are $R_{A\vee B\vee C}$ equivalent.  Front projection diagrams of the links $L$ and $L'$ are shown in Figure \ref{Fig:Chain}.

\subsection{The proof of Theorem \ref{Thm:Kirby}}\label{Sec:KirbyProof}

In this section we prove Theorem \ref{Thm:Kirby}.  Throughout $\Mxi$ will be a fixed contact manifold.  $X=X^{+}\cup X^{-}$ and $Y=Y^{+}\cup Y^{-}$ will denote surgery links in $\Sthree$, each of which determines $\Mxi$.

\begin{proof}[Proof of Theorem \ref{Thm:Kirby}]

As stated in the introduction, we would like to use Theorem \ref{Thm:TwistSurgery} to interpret this result as Theorem \ref{Thm:GirCor} for mapping classes with specified Dehn twist factorizations.  With this motivation, the proof of Theorem \ref{Thm:Kirby} will be broken up into the following steps:
\be
\item Describe positive stabilization and monodromy conjugation operations for surgery diagrams.  Show that these operations can be recovered by Legendrian isotopies and ribbon moves.
\item Stabilize $X$ and $Y$ as in (1) to obtain surgery links $\widetilde{X}$ and $\widetilde{Y}$ such that there is an open book $(\widetilde{\Sigma},\Phi_{\widetilde{\Sigma}})$ supporting $\Sthree$, constructed via Algorithm 2 such that
    \be
    \item $\widetilde{X}$ and $\widetilde{Y}$ are contained in a neighborhood $\widetilde{\Sigma}\times[-1,1]$ of a page of $(\widetilde{\Sigma}, \Phi_{\widetilde{\Sigma}})$, and
    \item $\Phi_{\widetilde{\Sigma}}\circ D_{\widetilde{X}}$ is conjugate to $\Phi_{\widetilde{\Sigma}}\circ D_{\widetilde{Y}}$ by some diffeomorphism $\Psi\in MCG(\widetilde{\Sigma},\partial \widetilde{\Sigma})$.
    \ee
\item Describe $\Psi$ in terms of surgery as in (1).  This will relate $\widetilde{X}$ and $\widetilde{Y}$ by ribbon moves and Legendrian isotopies, and so will complete the proof.
\ee

\noindent\textbf{Step 1.} The positive stabilization operation for surgery diagrams was described in the proof of Theorem \ref{Thm:Destab}(1).  It is easy to see that the insertion of a positive stabilization can be performed by inserting a canceling pair, a Legendrian isotopy, and a handle slide as in Section \ref{Sec:HandleSlide}.  See Figure \ref{Fig:StableSlide}.  Now we must describe an operation for surgery diagrams analogous to conjugating the monodromy of an open book.  Throughout, $\epsilon$ will be an arbitrarily small positive constant.

\begin{figure}[h]
	\begin{overpic}[scale=.7]{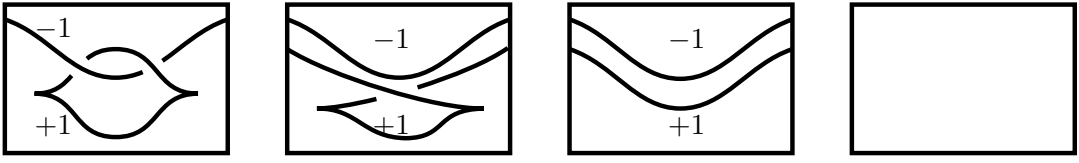}
        \put(3,2){$+1$}
        \put(3, 11){$-1$}
        \put(34.5,10){$-1$}
        \put(34.5,2){$+1$}
        \put(62,10){$-1$}
        \put(62,2){$+1$}
    \end{overpic}
    \vspace{1.5mm}
	\caption{Deleting a positive stabilization, as described in Theorem \ref{Thm:Destab}, can be performed by a sequence of ribbon moves and Legendrian isotopies.  From left to right we perform a handle slide (see Section \ref{Sec:HandleSlide}), then a Legendrian isotopy, and finally a deletion of a canceling pair.  The boxes indicate that we are working in a Darboux ball.}
    \label{Fig:StableSlide}
\end{figure}

\begin{defn}\label{Def:Conjugate}
Suppose that $\Sigma$ is the page of an open book $\AOB$ of $\Sthree$ and $L\subset\Sigma\times[-1+\epsilon,1-\epsilon]\subset\Sigma\times[-1,1]$ is a surgery link.  Let $K\subset\Sigma$ be a Legendrian realizable, simple, closed curve in $\Sigma$, and let $\delta\in\{
+1,-1\}$.  Define the surgery link $L_{(\Sigma,K^{\delta})}$ by
\begin{equation*}
L_{(\Sigma,K^{\delta})} = (\Phi^{-1}(K))^{\delta}(1)\;\cup\; L\;\cup\; K^{-\delta}(-1).
\end{equation*}
\end{defn}

\begin{lemma}\label{Lemma:Conjugate}
In the notation of the above definition, if $L\subset\Sthree$ presents $\Mxi$, then so does $L_{(\Sigma,K^{\delta})}$.
\end{lemma}

\begin{proof}
$\Mxi$ is supported by the open book $(\Sigma,\Phi\circ D_{L})$.  The contact manifold determined by $L_{(\Sigma,K^{\delta})}$ is supported by the open book with page $\Sigma$ and monodromy
\begin{equation*}
\begin{aligned}
\Phi\circ D^{-\delta}_{\Phi^{-1}(K)}\circ D(L)\circ D^{\delta}_{K} &=\Phi\circ\Phi^{-1}\circ D_{K}^{-\delta}\circ\Phi\circ D_{L}\circ D_{K}^{\delta} \\
&= D^{-\delta}_{K}\circ\Phi\circ D_{L}\circ D_{K}^{\delta}.
\end{aligned}
\end{equation*}
As these two mapping classes are conjugate, they determine the same contact manifold.  Hence, surgery on $L_{(\Sigma,K^{\delta})}$ produces $\Mxi$.
\end{proof}

Note that in the proof of the preceding lemma, it is essential that the surface $\Sigma$ used in Definition \ref{Def:Conjugate} is the page of an open book decomposition, and not just the ribbon of some Legendrian graph.

\begin{lemma}
The link $L_{(\Sigma,K^{\delta})}$ of Definition \ref{Def:Conjugate} can be obtained from $L$ by the insertion of a canceling pair and a Legendrian isotopy.
\end{lemma}

\begin{proof}
Let $\AOB$ be the open book of $\Sthree$ used to define $L_{(\Sigma,K^{\delta})}$.  This manifold can be described by
\begin{equation*}
\begin{gathered}
(\Sigma\times[-1-\epsilon,1+\epsilon])/\sim\quad\text{where}\\
(x,1+\epsilon)\sim(\Phi(x),-1-\epsilon)\;\forall x\in\Sigma\quad\text{and}\quad (x,t)\sim(x,t')\;\forall x\in\partial\Sigma;\;t,t'\in[-1-\epsilon,1+\epsilon]
\end{gathered}
\end{equation*}
where $L\subset \Sigma\times[-1+\epsilon,1-\epsilon]$.  Consider the canceling pair $K^{-\delta}(-1)\cup K^{\delta}(-1-\epsilon/2)$.  Applying the flow of the negative Reeb vector field of a contact form compatible with $\AOB$, we can isotop $K^{\delta}(-1-\epsilon/2)$ to live in the page $\Sigma\times \{1\}$ in the complement of $L\cup K(-1)$.  The image of $K^{\delta}(-1-\epsilon/2)$ under this isotopy will be $(\Phi^{-1}(K))^{\delta}(1)$.
\end{proof}

This completes Step 1.\\

\noindent\textbf{Step 2.}
Possibly after a Legendrian isotopy, we may assume that $X\subset\{x<0\}\subset\mathbb{R}^{3}$ and $Y\subset\{ x>0\}\subset\mathbb{R}^{3}$ and consider $X$ and $Y$ as being simultaneously embedded in $\mathbb{R}^{3}$.  Apply Algorithm 2 to $X\cup Y$ to obtain a Legendrian graph $G$ whose ribbon $\Sigma$ contains $X\cup Y$, and is the page of an open book $(\Sigma,\Phi_{\Sigma})$ supporting $\Sthree$.  Here $\Phi_{\Sigma}$ is determined by Theorem \ref{Thm:Mono}.  By Lemma \ref{Thm:ConnectedBinding} we may assume that the boundary of $\Sigma$ is connected.  Then $\Mxi$ is supported by the open books
\begin{equation*}
(\Sigma,\Phi_{\Sigma}\circ D_{X})\quad\text{and}\quad (\Sigma,\Phi_{\Sigma}\circ D_{Y}).
\end{equation*}

By Theorem \ref{Thm:GirCor}, these open books can be positively stabilized some number of times so that their monodromies will be conjugate.  We will assume that these stabilized open books have connected binding.  More precisely, there is a surface $\widetilde{\Sigma}$ with a single boundary component, two collections  $\{ \alpha_{j}\}_{1}^{2g}$ and $\{ \beta_{j} \}_{1}^{2g}$ of Legendrian realizable, simple, closed curves on $\widetilde{\Sigma}$, and a map $\Psi\in MCG(\widetilde{\Sigma},\partial\widetilde{\Sigma})$ such that
\begin{equation}\label{Eq:Conjugate}
\Phi_{\Sigma}\circ D_{X} \circ(\prod_{1}^{2g}D^{+}_{\alpha_{j}})= \Psi\circ\Phi_{\Sigma}\circ D_{Y} \circ(\prod_{1}^{2g}D^{+}_{\beta_{j}})\circ\Psi^{-1}.
\end{equation}
Here $\Phi_{\Sigma}$, $D_{X}$, and $D_{Y}$ extend to elements of $MCG(\widetilde{\Sigma},\partial\widetilde{\Sigma})$ via the inclusion $\Sigma\rightarrow\widetilde{\Sigma}$.

As both $\Sigma$ and $\widetilde{\Sigma}$ have connected boundary, we can write $\widetilde{\Sigma}$ as a boundary connected sum of $\Sigma$ and another surface $\Sigma'$ which has the topological type of a once-punctured genus $g$ surface.  We will embed $\widetilde{\Sigma}$ into $\Sthree$ so that it is the page of a supporting open book decomposition and extends the inclusion $\Sigma\subset\Sthree$.

Let $B$ be a Darboux ball in the complement of the 2-skeleton of the contact cell decomposition of $\Sthree$ associated to the Legendrian graph $G$ and open book $(\Sigma,\Phi_{\Sigma})$.  Embed $\Sigma'$ into $B$ via Algorithm 1.  Then $\Sigma'$ is the page of an open book $(\Sigma', \Phi_{\Sigma'})$ supporting $\Sthree$. Connect $\Sigma$ to $\Sigma'$ with the ribbon of a Legendrian arc as in Step 2 of Algorithm 2.  This gives rise to an embedding of $\widetilde{\Sigma}$ into $\Sthree$ as the page of a supporting open book $(\widetilde{\Sigma},\Phi_{\Sigma}\circ\Phi_{\Sigma'})$.

Step 2 of Algorithm 1 provides a surgery link $L\subset\Sigma'\times[-1,0]$ for which $D_{L}=\Phi_{\Sigma'}^{-1}$.  As $L$ is disjoint from $X\cup Y$ when projected to $\widetilde{\Sigma}$ we can Legendrian-isotop $L$ in $\Sigma'\times[-1,1]$ so that $L\subset\Sigma'\times[1/2,1]$ and the projection to $\Sigma'$ is unchanged.  Note that $D_{L}$ commutes with $D_{X}$ and $D_{Y}$.  Performing a similar isotopy we can move the link $X\cup Y$ from $\Sigma\times\{ 0\}$ to $\Sigma\times\{ 1\}$.

Now we can find surgery links $A$ and $B$ in $\widetilde{\Sigma}\times[-1,0]$ in the complement of $L\cup X\cup Y$ for which
\be
\item $D_{A}=\prod_{1}^{2g}D^{+}_{\alpha_{j}}$ and $D_{B}=\prod_{1}^{2g}D^{+}_{\beta_{j}}$,
\item the components of the link $A$ projected to $\widetilde{\Sigma}$ give the $\alpha_{j}$, and
\item the components of the link $B$ projected to $\widetilde{\Sigma}$ give the $\beta_{j}$.
\ee
Indeed, we can define
\begin{equation*}
\begin{gathered}
A=\alpha_{2g}^{-1}(0)\cup\alpha_{2g-1}^{-1}(\frac{-1}{2g})\cup\cdots\cup\alpha_{1}^{-1}(-1),\quad\text{and}\\
B=\beta_{2g}^{-1}(0)\cup\beta_{2g-1}^{-1}(\frac{-1}{2g})\cup\cdots\cup\beta_{1}^{-1}(-1).
\end{gathered}
\end{equation*}

It is easy to check by induction on $g$ that the embeddings of $\Sigma'$ can be chosen so that the surgery links $X\cup A\cup L$ and $Y\cup B\cup L$ account for $2g$ positive stabilizations - as described in Step 1 - of the surgery diagrams $X$ and $Y$.  Modifying the embedding of $\Sigma'$ may conjugate the mapping $\Psi$ of Equation \ref{Eq:Conjugate} but will leave $X$ and $Y$ unaffected.  Define
\begin{equation*}
\widetilde{X} := X\cup A\cup L\quad\text{and}\quad \widetilde{Y}:=Y\cup B\cup L.
\end{equation*}
Then $\widetilde{X}$ and $\widetilde{Y}$ are obtained from $X$ and $Y$, respectively from a sequence of ribbon moves and Legendrian isotopies.  Moreover, $\widetilde{X},\widetilde{Y}\subset\widetilde{\Sigma}\times[-1,1]$ and
\begin{equation*}
D_{\widetilde{X}}= \Phi_{\Sigma'}^{-1}\circ D_{X} \circ(\prod_{1}^{2g}D^{+}_{\alpha_{j}}),\quad\text{and}\quad D_{\widetilde{Y}}= \Phi_{\Sigma'}^{-1}\circ D_{Y} \circ(\prod_{1}^{2g}D^{+}_{\beta_{j}}).
\end{equation*}

This concludes Step 2 of the proof.\\

\noindent\textbf{Step 3.}
To finish the proof, we must show that $\widetilde{X}$ and $\widetilde{Y}$ are related by a sequence of ribbon moves and Legendrian isotopies.  By Lemma \ref{Lemma:Conjugate}, it suffices to show that we can modify $\widetilde{X}$ by a sequence of ``surgery conjugations'' as in Definition \ref{Def:Conjugate} to obtain a surgery link which is $\widetilde{\Sigma}$ equivalent to $\widetilde{Y}$.  This is what we will show.

Consider a collection of curves $\{\gamma_{j}\}_{1}^{k}$ on $\widetilde{\Sigma}$ which represent Lickorish generators of the mapping class group of $\widetilde{\Sigma}$.  Then each $\gamma_{j}$ is Legendrian realizable.  Express the mapping $\Psi$ of Equation \ref{Eq:Conjugate} as a product of Dehn twists on the $\gamma_{j}$:
\begin{equation*}
\Psi = D^{s_{1}}_{\zeta_{1}} \circ \cdots \circ D^{s_{l}}_{\zeta_{l}}
\end{equation*}
for some $l\in\mathbb{N}$, $s_{j}\in\{+,-\}$, and $\zeta_{j}\in\{\gamma_{j}\}$.  Now inductively define the surgery links
\begin{equation*}
\widetilde{X}_{0}:=\widetilde{X},\quad \widetilde{X}_{j}:=(\widetilde{X}_{j-1})_{(\widetilde{\Sigma},\zeta_{j}^{-s_{j}})}
\end{equation*}
using the notation of Definition \ref{Def:Conjugate}.  Then
\begin{equation*}
D_{\widetilde{X}_{l}} = \Psi^{-1}\circ \Phi_{\Sigma}\circ D_{X} \circ(\prod_{1}^{2g}D^{+}_{\alpha_{j}})\circ \Psi
\end{equation*}
implying that $\widetilde{X}_{l}$ is $\widetilde{\Sigma}$ equivalent to $\widetilde{Y}$ by Equation \ref{Eq:Conjugate}.
\end{proof}

\noindent\emph{Acknowledgements.}  I would like to thank my advisor Ko Honda for his support and guidance during the completion of this project, Patrick Massot for helping to correct some erroneous definitions -- in particular, Definition \ref{Def:ContactCell}(2) -- and attributions in a previous version of this paper, as well as an anonymous referee for their useful remarks.



\begin{thebibliography}{}

\bibitem[AO01]{AkOzb:LefFib}
S. Akbulut and B. Ozbagci, \textit{Lefschetz fibrations on compact Stein surfaces}, Geom. Topol., Vol. 5 (2001), 319-334

\bibitem[Al23]{Alexander:OB}
J. W. Alexander, \textit{A lemma on systems of knotted curves}, Proc. Nat. Acad. Sci. USA, Vol. 9 (1923), 93-95

\bibitem[Ar07]{Arikan:Genus}
M. Arikan, \textit{On the support genus of a contact structure}, J. G\"okova Geom. Topol., Vol. 1 (2007), 92-115

\bibitem[C04]{Calcut}
J. S. Calcut, \textit{Torelli actions and smooth structures on 4-manifolds}, Ph.D. thesis, University of Maryland (2004)

\bibitem[C05]{Calcut2}
J. S. Calcut, \textit{Knot theory and the Casson invariant in the Artin presentation theory} (Russian) Fundam. Prikl. Mat. 11 (2005), 119-126; translation in J. Math. Sci. (N. Y.) 144 (2007), 4446-4450

\bibitem[DG01]{DG:Surgery}
F. Ding and H. Geiges, \textit{Symplectic fillability of tight contact structures on torus bundles}, Algebr. Geom. Topol., Vol. 1 (2001), 153-172

\bibitem[DG04]{DG:LWall}
F. Ding and H. Geiges, \textit{A Legendrian surgery presentation of contact 3-manifolds}, Math. Proc. Cambridge Philos. Soc., Vol. 136 (2004), 583-598

\bibitem[DG09]{DG:Handles}
F. Ding and H. Geiges, \textit{Handle moves in contact surgery diagrams}, J. Topology, Vol. 2 (2009), 105-122

\bibitem[DGS05]{DGS:Lutz}
F. Ding, H. Geiges, and A. I. Stipsicz, \textit{Lutz twist and contact surgery}, Asian J. Math, Vol. 9 (2005), 57-64

\bibitem[E04]{Etnyre:Inter}
J. B. Etnyre, \textit{Planar open book decompositions and contact structures}, IMRN, Vol. 79 (2004), 4255-4267

\bibitem[E05]{Etnyre:KnotNotes}
J. B. Etnyre, \textit{Legendrian and transversal knots}, Handbook of knot theory, 105-185, Elsevier B. V., Amsterdam, 2005

\bibitem[E06]{Etnyre:OBIntro}
J. B. Etnyre, \textit{Lectures on open book decompositions and contact structures}, Floer homology, gauge theory, and low dimensional topology, Clay Math. Proc. Series 5 (2006), Amer. Math. Soc.,  Providence, RI, 103-141

\bibitem[EH01]{EtHo:KnotsI}
J. B. Etnyre and K. Honda, \textit{Knots and contact geometry I}, J. Symplectic Geom., Vol. 1 (2001), 63-120

\bibitem[EO08]{EtOzb:Support}
J. B. Etnyre and B. Ozbagci, \textit{Invariants of contact structures from open books}, Trans. Amer. Math. Soc., Vol. 360 (2008), 3133-3151

\bibitem[FR79]{FennRourke}
R. P. Fenn and C. P. Rourke, \textit{On Kirby's calculus of framed links}, Topology, Vol. 18 (1979), 1-15

\bibitem[Ge12]{Geiges:GeomTop}
H. Geiges, \textit{Contact structures and geometric topology}, Global Differential Geometry, Springer Proc. Math. 17, Springer-Verlag, Berlin (2012), 463-489

\bibitem[Gi91]{Giroux91}
E. Giroux, \textit{Convexit\'{e} en topologie de contact}, Comment. Math. Helv., Vol. 66 (1991), 637-677

\bibitem[Gi02]{Giroux:ContactOB}
E. Giroux, \textit{G\'{e}om\'{e}trie de contact: de la dimension trois vers les dimensions sup\'{e}rieures}, Proceedings of the International Congress of Mathematicians, Vol. II (2002), Higher Ed. Press, Beijing, 405-414

\bibitem[Gom98]{Gompf}
R. Gompf, \textit{Handlebody construction of Stein surfaces}, Ann. of Math. (2), Vol. 148 (1998) 619-693

\bibitem[Goo05]{Goodman:Sober}
N. Goodman, \textit{Overtwisted open books from sobering arcs}, Algebr. Geom. Topol., Vol. 5 (2005), 1173-1195

\bibitem[H00]{Honda:TightClassI}
K. Honda, \textit{On the classification of tight contact structures I}, Geom. Topol., Vol. 4 (2000), 309-386


\bibitem[HKM07]{HKM:RightVeering}
K. Honda, W. Kazez, and G. Mati\'{c}, \textit{Right-veering diffeomorphisms of compact surfaces with boundary}, Invent. Math., Vol. 169 (2007), 427-449

\bibitem[K78]{Kirby}
R. Kirby, \textit{A calculus for framed links in $S^{3}$}, Invent. Math., Vol. 45 (1978), 35-56

\bibitem[LiWa11]{LiWang}
Y. Li and J.Wang, \textit{The support genus of certain Legendrian knots}, preprint (2011), arXiv:1101.5213v1

\bibitem[L62]{Lickorish:Surgery}
W. B. R. Lickorish, \textit{A representation of orientable combinatorial 3-manifolds}, Annals of Math., Vol. 76 (1962), 531-540

\bibitem[L64]{Lickorish:Generators}
W. B. R. Lickorish, \textit{A finite set of generators for the homeotopy group of a 2-manifold}, Math. Proc. Cambridge Philos. Soc., Vol. 60 (1964), 769-778

\bibitem[LS04]{LS:Trefoil}
P. Lisca and A. I. Stipsicz, \textit{Ozsv\'{a}th-Szab\'{o} invariants and tight contact three-manifolds I}, Geom. Topol., Vol. 8 (2004), 925-945


\bibitem[NW10]{WendlNied:Stein}
K. Niederkr\"{u}ger and C. Wendl, \textit{Weak symplectic fillings and holomorphic curves}, Preprint (2010), arXiv:1003.3923v2


\bibitem[On10]{Ona:LegBook}
S. Onaran, \textit{Invariants of Legendrian knots from open book decompositions}, IMRN, No. 10 (2010), 1831-1859

\bibitem[Oz11]{OzbagciHandle}
B. Ozbagci, \textit{Contact handle decompositions}, Topology Appl., Vol. 158 (2011), 718-727. 

\bibitem[OS04]{OzbSt:SteinSurgery}
B. Ozbagci and A. I. Stipsicz, \textit{Surgery on contact 3-manifolds and Stein surfaces}, Bolyai Soc. Math. Studies 13 (2004), Springer-Verlag, Berlin

\bibitem[OSS]{OSS:PlanarFloer}
P. Ozsv\'{a}th, A. I. Stipsicz, and Z. Szab\'{o}, \textit{Planar open books and Floer homology}, IMRN, No.54 (2005), 3385-3401

\bibitem[P04]{Plame:Algorithm}
O. Plamenevskaya, \textit{Contact structures with distinct Heegaard Floer invariants}, Math. Res. Lett., Vol. 11 (2004), 547-561

\bibitem[R90]{Rolfsen}
D. Rolfsen, \textit{Knots and links} (Corrected reprint of the 1976 original), Mathematics Lecture Series, Vol. 7, Publish or Perish Inc. (1990), Houston, TX

\bibitem[S05]{Stipsicz}
A. I. Stipsicz, \textit{Surgery diagrams and open book decompositions of contact 3-manifolds}, Acta Math. Hungar, Vol. 108 (2005), 71-86

\bibitem[TW75]{ThWi:OB}
W. P. Thurston and H. E. Winkelnkemper, \textit{On the existence of contact forms}, Proc. Amer. Math. Soc., Vol. 52 (1975), 345-347

\bibitem[Wa10]{Wand}
A. Wand, \textit{Mapping class group relations, Stein fillings, and planar open book decompositions}, Preprint (2010), arXiv:1006.2550v2

\bibitem[Wei91]{Weinstein:Handles}
A. Weinstein, \textit{Contact surgery and symplectic handlebodies}, Hokkaido Math. J., Vol. 20 (1991), 241-251.

\bibitem[Wen10]{Wendl:ECH}
C. Wendl, \textit{A hierarchy of local symplectic filling obstructions for contact 3-manifolds}, Preprint (2010), arXiv:1009.2746v2



\end{thebibliography}
\end{document}